\documentclass{amsart}
\pdfoutput=1
\usepackage{amsmath}
\usepackage{amssymb}
\usepackage{amsthm}
\usepackage{amsmath}
\usepackage{array}
\usepackage{xy}
\usepackage{graphicx}

\DeclareGraphicsExtensions{.jpg,.pdf,.png,.jpeg}

\newtheorem{thm}{Theorem}[section]

\newtheorem{lem}[thm]{Lemma}

\newenvironment{remark}[1][Remark]{\begin{trivlist}
\item[\hskip \labelsep {\bfseries #1}]}{\end{trivlist}}

\newcommand{\lt}{{L^2}}

\newcommand{\htw}{{H^2}}

\newcommand{\Ho}{{H^1_0(\Omega)}}
\newcommand{\Htw}{{H^2(\Omega)}}
\newcommand{\eps}{\varepsilon}
\newcommand{\normd}[1]{\frac{\partial #1}{\partial \nu}}
\newcommand{\duel}{{(H^1_0\cap H^2)^*}}

\newcommand{\qnorm}{{H^{-\frac12}(\partial\Omega)}}
\newcommand{\spa}{\hspace{1cm}}
\newcommand{\um}{u^{\varepsilon,\mu}_h}
\newcommand{\ue}{u^\varepsilon}
\newcommand{\uh}{u^\varepsilon_h}

\newcommand{\vepsi}{{\varepsilon}}
\newcommand{\Ome}{{\Omega}}
\newcommand{\p}{{\partial}}

\newcommand{\Div}{{\mbox{\rm div}}}

\numberwithin{equation}{section}

\newtheorem{Remark}{Remark}[section]

\begin{document}

\title[GALERKIN METHODS FOR MONGE-AMP\`ERE EQUATIONS]{Analysis of
Galerkin Methods for the Fully Nonlinear Monge-Amp\`ere Equation$^\star$}

\author{Xiaobing Feng}
\address{Department of Mathematics, The University of Tennessee, Knoxville, TN 37996, U.S.A.
(xfeng@math.utk.edu).}

\author{Michael Neilan }
\address{
Department of Mathematics, The University of Tennessee, Knoxville, TN 37996, U.S.A.
(neilan@math.utk.edu). }

\thanks{$\star$This work was partially supported by the NSF grants DMS-0410266 and DMS-0710831.}



\begin{abstract}
This paper develops and analyzes finite element Galerkin and
spectral Galerkin methods for approximating viscosity solutions
of the fully nonlinear Monge-Amp\`ere equation
$\det(D^2u^0)=f \,(>0)$ based on the vanishing moment method
which was developed by the authors in \cite{Feng2,Feng1}. In this approach, the
Monge-Amp\`ere equation is approximated by the fourth
order quasilinear equation $-\varepsilon\Delta^2 u^\varepsilon +
\det{D^2u^\varepsilon} =f$ accompanied by appropriate boundary conditions.
This new approach allows one to construct convergent Galerkin numerical
methods for the fully nonlinear Monge-Amp\`ere equation
(and other fully nonlinear second order partial differential equations),
a task which has been impracticable before. In this paper,
we first develop some finite element and spectral Galerkin methods for
approximating the solution $u^\varepsilon$ of the regularized fourth
order problem. We then derive optimal order error
estimates for the proposed numerical methods. In particular,
we track explicitly the dependence of the error bounds on the
parameter $\vepsi$, for the error $u^\varepsilon-u^\varepsilon_h$.
Due to the strong nonlinearity of the underlying equation, the standard
perturbation argument for error analysis of finite element approximations
of nonlinear problems does not work here. To overcome
the difficulty, we employ a fixed point technique which strongly
makes use of the stability property of the linearized problem and
its finite element approximations.
Finally, using the Aygris finite element method as an example,
we present a detailed numerical study of the rates of convergence
in terms of powers of $\vepsi$ for the error $u^0-u_h^\vepsi$,
and numerically examine what is the ``best" mesh size $h$ in relation to $\vepsi$
in order to achieve these rates.
\end{abstract}

\maketitle

\section{Introduction}\label{sec-1}

Fully nonlinear partial differential equations (PDEs) are those equations
which are nonlinear in the highest order derivative(s) of the unknown function(s).
In the case of the second order equations, the general form of the fully nonlinear
PDEs is given by
\begin{equation}\label{generalPDE}
F(D^2 u^0, Du^0, u^0, x)=0,
\end{equation}
where $D^2 u^0(x)$ and $Du^0(x)$ denote respectively the Hessian
and the gradient of $u^0$ at $x\in \Omega\subset \mathbf{R}^n$.  $F$ is
assumed to be a nonlinear function in at least one of entries of $D^2 u^0$.
Fully nonlinear PDEs arise from many scientific and engineering
fields including differential geometry, optimal control, mass transportation,
geostrophic fluid, meteorology (cf. \cite{Caffarelli_Cabre95,Caffarelli_Milman99,
Gilbarg_Trudinger01,Fleming_Soner06,McCann_Oberman04} and the references therein).

In this paper we focus our attention on a prototypical fully nonlinear second
order PDE, the well-known Monge-Amp\`ere equation. Our goal is to develop
and analyze finite element and spectral Galerkin methods for approximating
viscosity solutions of the Dirichlet problem for the Monge-Amp\`ere equation
(cf. \cite{Gutierrez01}):
\begin{alignat}{2}\label{monge1}
\det(D^2 u^0)&=f &&\qquad\text{in }\Omega\subset \mathbf{R}^n,\\
\label{monge2} u^0&=g &&\qquad\text{on }\partial\Omega,
\end{alignat}
where $\Omega\subset \mathbf{R}^n$ is a convex domain with smooth or
piecewise smooth boundary $\p\Ome$. $\det(D^2 u^0(x))$ denotes the determinant
of $D^2 u^0(x)$.  Clearly, the Monge-Amp\`ere equation is a special case of
\eqref{generalPDE} with $F(D^2 u^0, Du^0, u^0, x)= \det(D^2 u^0)-f$.
We refer the reader to \cite{Benamou_Brenier00,Caffarelli_Milman99,
Gilbarg_Trudinger01,Gutierrez01} and the references therein for the
derivation, applications, and properties of the Monge-Amp\`ere equation.

For fully nonlinear second order PDEs, in particular for the Monge-Amp\`ere equation,
it is well-known that their Dirichlet problems do not have classical solutions
in general even $f$, $g$ and $\p\Ome$ are smooth, if $\Ome$ is not strictly convex.
(see \cite{Gilbarg_Trudinger01}). So it is imperative to develop some
weak solution theories for these problems. However, because of the
fully nonlinearity of these equations, unlike in the case of linear and quasilinear PDEs,
the usual weak solution theory based on integration by parts does not work for
fully nonlinear PDEs, hence, other ``nonstandard" weak solution theories
must be sought. In the case of the Monge-Amp\`ere equation,
the first such theory was due to A. D. Aleksandrov, who introduced a notion of
generalized solutions and proved the Dirichlet problem with $f>0$ has a unique
generalized solution in the class of convex functions (cf. \cite{Aleksandrov61},
also see \cite{Cheng_Yau77}). But major advances on analysis of problem
\eqref{monge1}-\eqref{monge2} has only been achieved many years later after the introduction
and establishment of the viscosity solution theory (cf. \cite{Caffarelli_Cabre95,
Crandall_Ishii_Lions92,Gutierrez01}). We recall that the notion of viscosity
solutions was first introduced by Crandall and Lions \cite{Crandall_Lions83}
in 1983 for the first order fully nonlinear Hamilton-Jacobi equations.
It was quickly extended to second order fully nonlinear PDEs, with dramatic
consequences in the wake of a breakthrough of Jensen's maximum
principle \cite{Jensen88} and the Ishii's discovery \cite{Ishii89}
that the classical Perron's method could be used to infer existence
of viscosity solutions. It should be noted that there also exist nonconvex
solutions to problem \eqref{monge1}-\eqref{monge2}, so if the convexity
requirement is dropped, solutions to problem \eqref{monge1}-\eqref{monge2} are
not unique. We shall refer this uniqueness property as {\em conditional uniqueness}.
Nonuniqueness or conditional uniqueness is often expected for the Dirichlet
problems of fully nonlinear second order PDEs. It should also be
noted that unlike in the case of fully nonlinear first order PDEs,
the terminology ``viscosity solution" loses its original meaning in the
case of fully nonlinear second order PDEs.

For the Dirichlet Monge-Amp\`ere problem \eqref{monge1}-\eqref{monge2}
with $f>0$, we recall that \cite{Gutierrez01} a convex function $u^0\in C^0(\overline{\Ome})$
satisfying $u^0=g$ on $\p \Ome$ is called a {\em viscosity subsolution (resp. viscosity
supersolution)} of \eqref{monge1} if for any $\varphi\in C^2$ there
holds $\det(D^2\varphi(x_0))\leq f(x_0)$ (resp. $\det(D^2\varphi(x_0))\geq f(x_0)$)
provided that $u^0-\varphi$ has a local maximum (resp. a local minimum)
at $x_0\in \Ome$. $u^0\in C^0(\overline{\Ome})$ is called a {\em viscosity solution}
if it is both a {\em viscosity subsolution} and a {\em viscosity supersolution}.
Form this definition we can see that the notion of viscosity
solutions is not variational and not defined using the more familiar
integration by parts approach which in fact is not possible because
of the fully nonlinearity of the PDE. On the other hand, the
strategy of shifting derivatives onto the test functions at extremal
points in the definition of viscosity solutions can be viewed as a
``{\em differentiation by parts}" approach. It can be shown that a
viscosity solution satisfies the PDE in the classical sense at every
point where the viscosity solution has second order continuous
derivatives. However, the theory does not tell what equations or
relations the viscosity solution satisfies at other points.
Numerically, this non-variational nature of the viscosity solution
theory poses a daunting challenge for computing viscosity solutions
because it makes it impossible for one to directly approximate
viscosity solutions using any Galerkin type numerical methods
including finite element, spectral and discontinuous Galerkin
methods, which all are based on variational formulations of PDEs. In
addition, it is extremely difficult (if all possible) to mimic
``{\em differentiation by parts}" approach at the discrete level, so
there seems have no hope to develop a discrete viscosity solution
theory.

To overcome the above difficulties, recently in \cite{Feng1,Feng2}
we introduced a new approach, called the vanishing moment method,
for establishing the existence of viscosity solutions for fully
nonlinear second order PDEs, in particular, for problem
\eqref{monge1}-\eqref{monge2}. In addition, this new approach gives
arise a new notion of weak solutions, called {\em moment solutions},
for fully nonlinear second order PDEs. Furthermore, the vanishing
moment method is constructive, so practical and convergent numerical
methods can be developed based on the approach for computing the
viscosity solutions of fully nonlinear second order PDEs such as
problem \eqref{monge1}-\eqref{monge2}. The main idea of the
vanishing moment method is to approximate a fully nonlinear second
order PDE by a quasilinear higher order PDE. The notion of moment
solutions and the vanishing moment method are natural
generalizations of the original definition of viscosity solutions
and the vanishing viscosity method introduced for the
Hamilton-Jacobi equations in \cite{Crandall_Lions83}. We now briefly
recall the definitions of moment solutions and the vanishing moment
method, and refer the reader to \cite{Feng1,Feng2} for a detailed
exposition.

Firstly, the vanishing moment method approximates the fully nonlinear
equation \eqref{generalPDE} by the following quasilinear fourth order PDE:
\begin{equation}\label{moment1}
-\vepsi\Delta^2 u^\eps+F(D^2 u^\eps, Du^\eps, u^\eps, x) =0\spa \ \ (\eps >0),
\end{equation}
which holds in domain $\Ome$. Suppose the Dirichlet boundary
condition $u^0=g$ is given on $\p\Ome$, then it is
natural to impose the same boundary condition on $u^\vepsi$, hence,
\begin{equation}\label{moment2}
u^\vepsi = g\qquad \text{on }\p\Ome.
\end{equation}
Secondly, in addition to boundary condition \eqref{moment2}
one more boundary condition must be imposed to ensure uniqueness of
solutions. In \cite{Feng1} we proposed to use one of the following boundary conditions:
\begin{equation} \label{bc2}
\Delta u^\vepsi=\vepsi,\quad \text{or}\quad
D^2 u^\vepsi\nu\cdot\nu=\vepsi \spa \text{on}\ \partial\Omega,
\end{equation}
where $\nu$ stands for the unit outward normal to $\p\Ome$. Although
both boundary conditions work well numerically, the first boundary
condition is more convenient for standard Galerkin type methods such
as finite methods, spectral and discontinuous Galerkin methods,
while the second boundary condition fits mixed finite element
methods (cf. \cite{Feng1,Feng3}) better. Hence, in this paper we
shall use the first boundary condition. We also refer the reader to
\cite{Feng2} for the heuristic argument why these boundary
conditions were chosen in the first place.

To sum up, the vanishing moment method consists of approximating
second order boundary value problem
\eqref{monge2}--\eqref{generalPDE} by fourth order boundary value
problem \eqref{moment1}--\eqref{moment2}, \eqref{bc2}. In the case
of the Monge-Amp\`ere equation, this then results in approximating
boundary value problem \eqref{monge1}--\eqref{monge2} by the
following problem:
\begin{alignat}{2}\label{pdeequation1}
-\vepsi\Delta^2 u^\eps +\text{det}(D^2 u^\eps)&=f \spa &&\text{in}\ \Omega,\\
u^\eps&=g\spa &&\text{on}\ \partial\Omega, \label{pdeequation2} \\
\Delta u^\eps&=\eps &&\text{on}\ \partial\Omega.
\label{pdeequation3}
\end{alignat}
It was proved in \cite{Feng1} that if $f\geq 0$ in $\Ome$ then
problem \eqref{pdeequation1}--\eqref{pdeequation3} has a unique
solution $u^\vepsi$ which is a convex function over $\Ome$.
Moreover, $u^\vepsi$ uniformly converges as $\vepsi\to 0$ to the
unique viscosity solution of \eqref{monge1}--\eqref{monge2}. As a
byproduct, this also shows that \eqref{monge1}--\eqref{monge2} has
a unique moment solution which coincides with the unique viscosity
solution. Furthermore, it was proved that there holds the following
a priori bounds which will be used frequently later in this paper:
\begin{equation}\label{momentbounds}
\|u^\eps\|_{H^j}=O\bigl(\eps^{-\frac{j-1}{2}}\bigr), \quad
\|u^\eps\|_{W^{2,\infty}}=O\bigl(\frac{1}{\eps}\bigr), \quad
\|\text{cof}(D^2u^\eps)\|_{L^\infty}=O\bigl(\frac{1}{\eps}\bigr)
\end{equation}
for $j=2,3$. Where $\text{cof}(D^2u^\eps)$ denotes the cofactor matrix of the
Hessian $D^2u^\eps$.
With the help of the vanishing moment method, the
difficult task of computing the unique convex viscosity solution
of the fully nonlinear second order Monge-Amp\`ere
problem \eqref{monge1}--\eqref{monge2}, which has multiple
solutions (i.e. there are non-convex solutions), is now reduced to
a feasible task of computing the unique regular solution of the
quasilinear fourth order problem
\eqref{pdeequation1}--\eqref{pdeequation3}. In particular,
this allows one to use and/or adapt the wealthy amount of existing
numerical methods, in particular, finite element methods to solve
problem \eqref{monge1}--\eqref{monge2} via problem
\eqref{pdeequation1}--\eqref{pdeequation3}.

The specific goal of this paper is to develop and analyze Galerkin
methods for approximating the solution of \eqref{pdeequation1}--\eqref{pdeequation3}
in $2$-D and $3$-D. When deriving error estimates of
the proposed numerical methods, we are particularly interested in
obtaining error bounds that show explicit dependence on $\vepsi$.
Aygris confirming finite element method in $2$-D and Legendre
spectral Galerkin method in both $2$-D and $3$-D are specifically
considered in the paper although our analysis applies to any conforming
Galerkin method for problem \eqref{pdeequation1}--\eqref{pdeequation3}.
We note that finite element approximations of fourth order PDEs, in particular, the
biharmonic equation, were carried out extensively in 1970's in the
two-dimensional case (see \cite{Ciarlet78} and the references
therein),and have attracted renewed interests lately for
generalizing the well-know $2$-D finite elements to the $3$-D case
(cf. \cite{Wang_Shi_Xu07,Wang_Xu07,Tai_Wagner00}) and for
developing discontinuous Galerkin methods in all dimensions (cf.
\cite{Feng_Karakashi07,Suli03}). Clearly, all these methods can be
readily adapted to discretize problem \eqref{pdeequation1}--\eqref{pdeequation3}
although their convergence analysis do not come easy due to the strong
nonlinearity of the PDE \eqref{pdeequation1}. Also, the standard
perturbation technique for deriving error estimates for numerical
approximations of mildly nonlinear problems does not work for
problem \eqref{pdeequation1}--\eqref{pdeequation3}.  We refer the reader
to \cite{Feng3,Neilan_thesis} for further discussions in this direction.

We also like to mention that a few results on numerical approximations of the
Monge-Amp\`ere equation as well as related equations have recently been 
reported in the literature.  Oliker and Prussner \cite{Oliker_Prussner88}
constructed a finite difference scheme for computing Aleksandrov measure induced
by $D^2u$ in $2$-D and obtained the solution $u$ of problem
\eqref{pdeequation1}--\eqref{pdeequation3} as a by-product.
Baginski and Whitaker \cite{Baginski_Whitaker96}
proposed a finite difference scheme for Gauss curvature equation
(cf. \cite{Feng2} and the references therein) in $2$-D by mimicking the unique
continuation method (used to prove existence of the PDE) at the discrete level.
In a series of papers (cf. \cite{Dean_Glowinski06b} and the references therein)
Dean and Glowinski proposed an augmented Lagrange multiplier method
and a least squares method for problem \eqref{pdeequation1}--\eqref{pdeequation3}
and the Pucci's equation (cf. \cite{Caffarelli_Cabre95,Gilbarg_Trudinger01}) in $2$-D
by treating the Monge-Amp\`ere equation and Pucci's equation as a constraint
and using a variational criterion to select a particular solution.
Very recently, Oberman \cite{Oberman07} constructed some wide stencil finite
difference scheme which fulfill the convergence criterion established by Barles
and Souganidis in \cite{Barles_Souganidis91} for finite difference approximations of
fully nonlinear second order PDEs. Consequently, the convergence
of the proposed wide stencil finite difference scheme immediately follows
from the general convergence framework of \cite{Barles_Souganidis91}.
Numerical experiments results were reported in \cite{Oliker_Prussner88,Oberman07,
Baginski_Whitaker96,Dean_Glowinski06b}, however, convergence analysis was not
addressed except in \cite{Oberman07}.

The remainder of this paper is organized as follows. In Section
\ref{sec-2}, we first derive the weak formulation for problem
\eqref{pdeequation1}-\eqref{pdeequation3} and then present our
confirming finite element and spectral Galerkin methods based on
this weak formulation. In Section \ref{sec-3}, we study the
linearization of problem \eqref{pdeequation1}--\eqref{pdeequation3}
and its Galerkin approximations. The results of this section, which
are of independent interests in themselves, will play an important
role in our error analysis for the numerical method introduced in
Section \ref{sec-2}. In Section \ref{sec-4}, we establish optimal
order error estimates in the energy norm for the proposed confirming
finite element and spectral Galerkin methods. Our main ideas are to
use a fixed point technique and to make strong use of the stability
property of the linearized problem which is analyzed in Section
\ref{sec-3}. In addition, we derive the optimal order error estimate
in the $L^2$-norm for $u^\vepsi-u^\vepsi_h$ using a duality
argument. Finally, in Section \ref{sec-5}, we first run some numerical
tests to validate our theoretical error estimate results. We then
present a detailed computational study for determining the ``best''
choice of mesh size $h$ in terms of $\varepsilon$ in order to
achieve the optimal rates of convergence and for estimating the
rates of convergence for both $u^0-u^\eps_h$ and $u^0-u^\eps$ in
terms of powers of $\eps$.

\section{Formulation of Galerkin Methods}\label{sec-2}
Standard space notations are adopted in this paper, we refer to
\cite{Brenner,Gilbarg_Trudinger01,Ciarlet78} for their exact
definitions. In addition, $\Ome$ denotes a bounded domain in
$\mathbf{R}^n$. $(\cdot,\cdot)$ and $\langle\cdot, \cdot\rangle$
denote the $L^2$-inner products on $\Ome$ and on $\p\Ome$,
respectively. $C$ is used to denote a generic $\vepsi$-independent
positive constant.  We also introduce the following special space
notation:
\[
V:=H^2(\Omega),\spa V_0:=H^2(\Omega)\cap H^1_0(\Omega),\spa
V_g:=\{v\in V;\ v|_{\partial\Omega}=g\}.
\]

Testing \eqref{pdeequation1} with $v\in V_0$ yields
\begin{equation}\label{variational1}
-\eps\int_\Omega \Delta u^\eps\Delta
v dx +\int_\Omega \text{det}(D^2u^\eps)v dx=\int_\Omega fv
dx-\int_{\partial\Omega}\eps^2\normd{v}ds.
\end{equation}
Based on \eqref{variational1}, we define the variational
formulation of \eqref{pdeequation1}--\eqref{pdeequation3} as follows:
Find $u^\eps\in V_g$ such that
\begin{equation}\label{variational2}
-\eps(\Delta u^\eps,\Delta v)+(\text{det}(D^2u^\eps),v)=(f,v)-\left\langle
\eps^2,\normd{v}\right\rangle_{\partial\Omega}\spa \forall v\in V_0.
\end{equation}

\begin{Remark}
We note that $\text{\rm det}(D^2 u^\eps)=\frac1n \Phi^\eps
D^2u^\eps=\frac1n\displaystyle\sum_{i=1}^n \Phi^\eps_{ij}u_{x_ix_j}\
\ j=1,2,...n.$, where $\Phi^\eps$ is the cofactor matrix of $D^2
u^\eps$. Thus, using the divergence free property of the cofactor
matrix $\Phi^\eps$ (cf. Lemma \ref{lem3.1}) we can define the
following alternative variational formulation for
\eqref{variational2}.
\begin{align}\label{altvar}-\eps(\Delta u^\eps,\Delta
\psi)-\frac1n(\Phi^\eps Du^\eps,D\psi)&=\langle f,\psi\rangle -
\left\langle \eps^2,\normd{\psi}\right\rangle_{\partial\Omega}\spa
\forall \psi\in V_0.\end{align} However, we shall not use the above
weak formulation in this paper although it is interesting to compare
Galerkin methods based on the two different but equivalent weak
formulations.
\end{Remark}

In this paper, we shall consider two types of Galerkin
approximations for \eqref{variational2}. The first type is the
confirming finite element Galerkin method in $2$-D.  Aygris finite
element will be used as a specific example of this class of methods
although our analysis is applicable to other confirming finite
element such as Bell element, Bogner--Fox--Schmit element, and
Hsieh--Clough--Tocher element (cf. \cite{Ciarlet78}). The second
type of the methods is the spectral Galerkin method in $2$-D and
$3$-D. Legendre spectral Galerkin method will be used as a specific
example although our analysis also applies to other spectral methods
(with necessary assumptions on the domain in the case of Fourier 
spectral method).

To formulate the finite element method in $2$-D, let $T_h$ be a
quasiuniform triangular or rectangular partition of
$\Ome \subset \mathbf{R}^2$ with mesh size $h\in (0,1)$ and $V^h\subset
V$ denote a confirming finite element space consisting of piecewise
polynomial functions of degree $r (\geq 5)$ such that for any
$v\in V\cap H^s(\Ome)$
\begin{align}
\inf_{v_h\in V^h} \|v-v_h\|_{H^j} \leq h^{\ell-j} \|v\|_{H^s},\quad
j=0,1,2;\, \ell=\min\{r+1,s\}.
\end{align}
We recall that $r=5$ in the case of Aygris element (cf. \cite{Ciarlet78}).

Let
\begin{align*}
V^h_g=\{v_h\in V^h;\ v_h|_{\partial\Omega}=g\},\spa
V^h_0=\{v_h\in V^h;\ v_h|_{\partial\Omega}=0\}.
\end{align*}
Based on the weak formulation \eqref{variational2}, we define our
finite element Galerkin method as follows. Find $u^\eps_h\in V^h_g$ such that
\begin{equation}\label{galmethod}
-\eps\bigl(\Delta u^\eps_h,\Delta v_h\bigr)+\bigl(\text{det}(D^2u^\eps_h),v_h\bigr)
=(f,v_h)-\left\langle \eps^2,\normd{v_h}\right\rangle_{\partial\Omega}\quad
\forall v_h\in V^h_0.
\end{equation}

To formulate the spectral Galerkin method, we assume that $\Ome$ is a
rectangular domain and let $L_j$ denote the
$j$th order Legendre polynomial of single variable and define the
following finite dimensional spaces: For $N\geq 2$, let
\begin{align*}
&V^N:=\text{span}\{L_0(x_1),L_2(x_1), \cdots, L_N(x_1)\}\quad\text{when } n=1,\\
&V^N:=\text{span}\{L_i(x_1) L_j(x_2);\, i,j=1,2,\cdots,N\} \quad\text{when } n=2,\\
&V^N:=\text{span}\{L_i(x_1) L_j(x_2) L_k(x_3);\, i,j,k=1,2,\cdots,N\}
\quad\text{when } n=3.
\end{align*}
It is well-known that $V^N$ has the following approximation property
(cf. \cite{Bernardi_Maday97}):
\begin{align}
\inf_{v_N\in V^N} \|v-v_N\|_{H^j} \leq CN^{j-s}
\|v\|_{H^s},\quad 0\leq j\leq \min\{s, N+1\}
\end{align}
for any $v\in V\cap H^s(\Ome)$.

Again, based on the weak formulation \eqref{variational2}, our spectral
Galerkin method for \eqref{variational2} is defined as
seeking $u^\vepsi_N\in V^N\cap V_g$ such that for any $v_N\in V^N\cap V_0$
\begin{equation}\label{spectralmethod}
-\eps\bigl(\Delta u^\eps_N,\Delta v_N\bigr)+\bigl(\text{det}(D^2u^\eps_N),v_N\bigr)
=(f,v_N)-\left\langle \eps^2,\normd{v_N}\right\rangle_{\partial\Omega}.
\end{equation}

Clearly, Galerkin methods \eqref{galmethod} and
\eqref{spectralmethod} have the exact same form as the variational
problem \eqref{variational2}. The only difference is that the
infinite dimensional space $V$ in \eqref{variational2} is replaced
by the finite dimensional subspace $V^h$ and $V^N$, respectively.
Letting $h:=\frac{1}{N}$, $u^h:=u^N$ and $V^h:=V^N$ in the
definition of the spectral method, \eqref{spectralmethod} can be
rewritten as the exactly same form as \eqref{galmethod}. For this
reason, we shall abuse the notation by using $u^h$ to denote the
solution of \eqref{galmethod} and the solution of
\eqref{spectralmethod} with understanding that $h=\frac{1}{N}$.
Since the convergence analyses for \eqref{galmethod} and for
\eqref{spectralmethod} are essentially same, we shall only present
the detailed analysis for \eqref{galmethod} and make comments about
that for \eqref{spectralmethod} in case there is a meaningful
difference.

Let $u^\eps$ be the solution to \eqref{variational2} and
$u^\eps_h$ the solution of \eqref{galmethod} or \eqref{spectralmethod}.
As mentioned in Section \ref{sec-1}, the main task of this paper is to derive
optimal error estimates for $u^\eps-u^\eps_h$. To this end, we
first need to prove existence and uniqueness of $u^\eps_h$. It
turns out both tasks are not easy due to the strong
nonlinearity in \eqref{galmethod} and \eqref{spectralmethod}.
Unlike the continuous PDE case where $u^\vepsi$ is proved to be convex
for all $\vepsi$ (cf. \cite{Feng2}), it is not clear whether $u^\eps_h$
preserves the convexity even for small $h$. Without a guarantee of
convexity for $u^\eps_h$, we could not establish any stability
result for $u^\eps_h$. This is the main obstacle for proving existence and
uniqueness for \eqref{galmethod} and \eqref{spectralmethod}. In addition,
again due to the strong nonlinearity, the standard perturbation technique
for deriving error estimate for numerical approximations of mildly
nonlinear problems does not work here. To overcome the difficulty, our
idea is to adopt a combined fixed point and linearization technique
which was used by the authors in \cite{Feng_Neilan_Prohl07}, where
a nonlinear singular second order problem known as the inverse mean
curvature flow was studied. We note that by using this technique we are
able to simultaneously prove existence and uniqueness for $u^\eps_h$ and
also derive the desired error estimates. In the next two sections,
we shall give the detailed account of this technique and apply it
to problems \eqref{galmethod} and \eqref{spectralmethod}.

\section{Linearization and its Finite Element Approximation}\label{sec-3}

To analyze \eqref{galmethod} and \eqref{spectralmethod}, we
shall study the linearization of \eqref{pdeequation1} to
establish the required technical tools. First, we recall the following
divergence-free row property for the cofactor matrices, which will
be used frequently in later sections.  We refer the reader to
\cite[p. 440]{evans} for a short proof of the lemma.
\begin{lem}\label{lem3.1}
Given a vector-valued function $\mathbf{v}=(v_1,v_2,\cdots,v_n):
\Ome\rightarrow \mathbb{R}^n$. Assume $\mathbf{v}\in
[C^2(\Ome)]^n$. Then the cofactor matrix $\text{\rm
cof}(D\mathbf{v})$ of the gradient matrix $D\mathbf{v}$ of
$\mathbf{v}$ satisfies the following row divergence-free property:
\begin{equation}\label{e3.1}
\Div (\text{\rm cof}(D\mathbf{v}))_i =\sum_{j=1}^n \p_{x_j}
(\text{\rm cof}(D\mathbf{v}))_{ij} =0 \qquad\text{\rm for }
i=1,2,\cdots, n,
\end{equation}
where $(\text{\rm cof}(D\mathbf{v}))_i$ and $(\text{\rm
cof}(D\mathbf{v}))_{ij}$ denote respectively the $i$th row and the
$(i,j)$-entry of $\text{\rm cof}(D\mathbf{v})$.
\end{lem}

\subsection{Linearization}

It is easy to check that for a given smooth function $w$ there holds
\begin{equation}
\text{det}(D^2u^\eps+tw)=\text{det}(D^2 u^\eps)
+t\text{tr}(\Phi^\eps D^2 w)+\cdots +t^n\text{det}(D^2 w).
\end{equation}
The linearization of $M^\eps(u^\vepsi):=-\eps\Delta^2u^\eps+\text{det}(D^2u^\eps)$
at the solution $u^\eps$ is given by
\begin{align*}
L_{u^\eps}(w):&=\lim_{t\to 0} \frac{M^\eps(u^\vepsi+tw) - M^\vepsi(u^\vepsi)}{t}\\
&= -\eps\Delta^2 w + \Phi^\eps: D^2w =-\eps\Delta^2 w +
\text{div}(\Phi^\eps Dw),
\end{align*}
where $\Phi^\eps$ denotes the cofactor matrix of $D^2u^\vepsi$ and
we have used Lemma \ref{lem3.1} to get the last equality. Also,
using Lemma \ref{lem3.1} it is easy to check that $L_{u^\eps}$ is
self-adjoint, i.e., the adjoint operator $L_{u^\eps}^*$ of $L_{u^\eps}$
coincides with $L_{u^\eps}$.

We now consider the following linear problem.
\begin{alignat}{2}\label{lin}
L_{u^\varepsilon}(v)&=\varphi  &&\quad \text{in}\ \Omega,\\
v&=0  &&\quad \text{on}\ \partial\Omega, \label{lin2}\\
\Delta v&=q &&\quad \text{on}\ \partial\Omega.\label{lin3}
\end{alignat}
Multiplying the PDE by a test function $w\in V_0$ and integrating over $\Omega$ we
get the following weak formulation for \eqref{lin}--\eqref{lin3}:
Find $v\in V_0$ such that
\begin{equation}\label{weaklin}
B[v,w]=\langle \varphi,w\rangle
-\eps\left\langle q,\normd{w}\right\rangle_{\partial\Omega}\spa \forall w\in V_0,
\end{equation}
where
\[
B[v,w]:=\eps\int_{\Omega}\Delta v\Delta w\ dx
+ \int_{\Omega} \Phi^\eps Dv\cdot Dw\ dx.
\]

The next theorem ensures the well-posedness of the above variational problem.
\begin{thm}\label{linbound}
Suppose $\partial\Omega\in \mathcal{C}^2$. Then for every
$\varphi\in V^*_0$ and $q\in \qnorm$ there exists a unique solution $v\in V_0$
to problem \eqref{weaklin}.
Furthermore, there exists a positive constant $C_1(\eps)$ such that
\begin{equation}\label{estimate}
\|v\|_{H^{2}}\le C_1(\eps)\left(\|\varphi\|_\duel+\|q\|_\qnorm\right).
\end{equation}
\end{thm}

\begin{proof}
It is easy to check that $B[\cdot,\cdot]$ is a bounded and coercive
bilinear form on $V_0\times V_0$ with coercive constant
$C_2(\eps)=O(\eps)$. Also, the right-hand side of \eqref{weaklin}
defines a bounded linear functional on $V_0$ (see
\cite{Neilan_thesis} for details). So the assertions of the theorem
follows immediately from an application of Lax-Milgram Theorem
\cite{Gilbarg_Trudinger01}.
\end{proof}

For more regular data, the above theorem can be improved to the following one
(see \cite{Grisvard85} for a similar proof).
\begin{thm}\label{linbound2} Suppose $\varphi\in (H^{s}(\Omega)\cap \Ho)^*,\
q\in H^{s-\frac52}(\partial\Omega),$ $\partial\Omega\in
\mathcal{C}^{s},\ (s\ge 2)$ and v is the unique solution to
\eqref{weaklin}.  Then $v\in H^{s}(\Omega)\cap \Ho$, and there exists
a $C_s(\eps)>0$ such that
\begin{equation}\label{estimate2}
\|v\|_{H^{s}}\le C_s(\eps)\left(\|\varphi\|_{(H^s\cap\Ho)^*}
+\|q\|_{H^{s-\frac52}(\partial\Omega)}\right).
\end{equation}
\end{thm}

\begin{Remark}
From the definition of $C_2(\eps)$, we see that in Theorem $\ref{linbound}$,
$C_1(\eps)=\textsl{O}\left(\frac{1}{\eps}\right)$. Currently, the
explicit dependence of $C_s(\eps)$ on $\vepsi$ in Theorem $\ref{linbound2}$
remains unknown for general $s\ge 3$. Finally, we note
that Theorem \ref{linbound} can easily be extended to the case of
nonhomogeneous boundary data which we summarize in the following theorem.
\end{Remark}

\begin{thm}\label{nonhomo}
For every $\varphi\in V^*,\ g\in H^{\frac32}(\partial\Omega),\ q\in \qnorm$,
there exists a unique weak solution $w\in V$ to
\begin{alignat*}{2}
L_{u^\varepsilon}(w)&=\varphi &&\quad \text{\rm in}\ \Omega,\\
w&=g &&\quad \text{\rm on}\ \partial\Omega, \\
\Delta w&=q &&\quad \text{\rm on}\ \partial\Omega.
\end{alignat*}
Furthermore, there exists $C>0$ such that
\begin{equation}\label{estimate3}
\|w\|_{H^{2}} \le \frac{C}{\eps}\left(\|\varphi\|_\duel
+\|g\|_{H^{\frac32}(\partial\Omega)}+\|q\|_{\qnorm}\right).
\end{equation}
\end{thm}

We refer the reader to \cite{Grisvard85} for a similar proof.


\subsection{Finite element approximation of linearized problem}\label{sec-3.2}

Let $V_0^h\subset V_0$ be one of the finite-dimensional subspace of
$V_0$ as defined in the previous subsection (e.g. Aygris finite
element and Legendre spectral element), and $v\in V_0$ denote the
solution of \eqref{weaklin}.  Based on the variational equation
$\eqref{weaklin}$, our Galerkin method for $\eqref{lin}$ is defined
as seeking $v_h\in V^h_0$ such that
\begin{equation}\label{weakfem}
B[v_h,w_h]=\langle \varphi, w_h\rangle - \eps \left\langle q,
\normd{w_h}\right\rangle_{\partial\Omega}\spa \forall w_h\in V^h_0.
\end{equation}

Our objective in this subsection is to first prove existence and
uniqueness for problem \eqref{weakfem} and then derive optimal
order error estimates in various norms.

\begin{thm}\label{existence}
Suppose $v\in V_0\cap H^s(\Omega)\ (s\ge 3).$  Then there exists a
unique $v_h\in V^h_0$ satisfying \eqref{weakfem}.
Furthermore, we have the following estimates:
\begin{align}\label{nest}
\|v_h\|_{H^{2}(\Omega)}&\le C_3(\eps) \left(\|\varphi\|_\duel+\|q\|_\qnorm\right),
\\
\|v-v_h\|_\Htw &\le C_4(\eps) h^{\ell-2}\|v\|_{H^{\ell}(\Omega)}, \label{nest2}\\
\|v-v_h\|_{H^1(\Omega)}&\le C_5(\eps) h^{\ell-1}\|v\|_{H^{\ell}(\Omega)},
\label{nest3}\\
\|v-v_h\|_{L^2(\Omega)}&\le C_6(\eps) h^{\ell}\|v\|_{H^{\ell}(\Omega)},\label{nest4}
\end{align}
where $\ell=\text{\rm min}\{r+1,s\}$ in the case of the finite element Galerkin 
method, and $\ell=\text{\rm min}\{N+1,s\}$ in the case of the spectral Galerkin method.
\end{thm}

\begin{proof}
Estimate \eqref{nest} follows immediately from setting $w_h=v_h$ in
\eqref{weakfem} and using the coercivity of the bilinear form $B[\cdot,\cdot]$.

To derive the error estimate in the $H^2$-norm, we use the error equation,
\[
B[v-v_h,w_h]=0\spa \forall w_h\in V^h_{0}.
\]
Using the coercivity property of $B[\cdot,\cdot]$, we have
\begin{align}
C_2(\vepsi)\|v-v_h\|^2_{H^2}&\le B[v-v_h, v-v_h]=B[v-v_h,v]-B[v-v_h,v_h]\\
&=B[v-v_h,v]=B[v-v_h,v-I_h v], \nonumber
\end{align}
where $C_2(\vepsi)=O(\vepsi)$, $I_h v$ denotes the Galerkin interpolant
of $v$ onto $V^h_0$. On noting that
\[
B(v-v_h,v-I_h v)\le \frac{C}{\sqrt{\eps}}
\|v-v_h\|_{H^2_0}\|v-I_h v\|_{H^2_0},
\]
we have
\[
\|v-v_h\|_{H^2}\le \frac{C}{C_2(\vepsi)\sqrt{\eps}}\|v-I_h v\|_{H^2}
\leq C_4(\eps) h^{l-2}\|v\|_{H^l},
\]
where
$C_4(\eps)=O\left(\eps^{-\frac32}\right).$  Thus, \eqref{nest2} holds.

Next, we derive the $H^1$-norm error estmate using a duality argument. Define
$e_h:=v-v_h$ and consider the following problem:
\begin{alignat*}{2}
L_{u^\varepsilon}(\psi)&=\Delta e_h\spa  &&\text{in}\ \Omega,\\
\psi&=0 \spa &&\text{on}\ \partial\Omega,\\
\Delta \psi&=0  && \text{on}\ \partial\Omega.
\end{alignat*}
Using $\eqref{estimate2}$, we have
\[
\|\psi\|_{H^{3}}\le C_s(\eps)\|\Delta e_h\|_{H^{-1}}.
\]

Since $\|\Delta e_h\|_{H^{-1}}=\text{sup}\{\langle \Delta e_h,
u\rangle |\ u\in H^1_0(\Omega),\ \|u\|_{H^1_0}\le 1\}$, we have
\[
\langle \Delta e_h,u\rangle = (\nabla e_h,\nabla u)
\le \|\nabla e_h\|_\lt\|\nabla u\|_\lt
\le \|\nabla e_h\|_{L^2} \|u\|_{H^1_0}=\|\nabla e_h\|_\lt .
\]
It follows that
\[
\|\Delta e_h\|_{H^{-1}}\le \|\nabla e_h\|_{L^2}.
\]
Thus,
\begin{align*}
\|\nabla e_h\|_{L^2}^2&=\langle \Delta e_h, e_h\rangle
=B[e_h,\psi]=B[e_h,\psi-I_h\psi]\\
&\le \frac{C}{\sqrt{\eps}}\|\psi-I_h\psi\|_{H^2}\|e_h\|_{H^2}\le
\frac{hC}{\sqrt{\eps}} \|\psi\|_{H^3}\|e_h\|_{H^2} \nonumber\\
&\le \frac{hC_s(\eps)}{\sqrt{\eps}} \|\Delta
e_h\|_{H^{-1}}\|e_h\|_{H^2}\le
\frac{hC_s(\eps)}{\sqrt{\eps}}\|\nabla e_h\|_\lt\|e_h\|_{H^2}. \nonumber
\end{align*}
Hence,
\[
\|\nabla e_h\|_\lt \le \frac{hC_s(\eps)}{\sqrt{\eps}}\|e_h\|_{H^2}.
\]
Combining the above with \eqref{nest2}, we get \eqref{nest3} with
$C_5=O\left(C_s(\eps)\eps^{-2}\right).$

To derive the error in the $L^2$-norm, we consider the following problem:
\begin{alignat*}{2}
L_{u^\varepsilon}(\psi)&=e_h \spa &&\text{in}\ \Omega,\\
\psi&=0  \spa &&\text{on}\ \partial\Omega, \\
\Delta\psi&=0  \spa &&\text{on}\ \partial\Omega.
\end{alignat*}
On noting $\eqref{estimate2}$ implies that
\[
\|\psi\|_{H^{4}}\le C_s(\eps)\|e_h\|_{L^2}.
\]
Thus,
\begin{align*}
\|e_h\|_\lt^2&=(e_h,v-v_h)=B[v-v_h,\psi]=B[v-v_h,\psi-I_h\psi]\\&\le
\frac{C}{\sqrt{\eps}} \|v-v_h\|_{H^2}\|\psi-I_h\psi\|_{H^2}\\
&\le \frac{h^2C}{\sqrt{\eps}}\|v-v_h\|_{H^2}\|\psi\|_{H^{4}}\le
\frac{h^2 C_s(\eps)}{\sqrt{\eps}} h^2 \|v-v_h\|_{H^2}\|e_h\|_\lt .
\end{align*}
Dividing by $\|e_h\|_\lt$, we get \eqref{nest4} with
$C_6=O\left(C_s(\eps)\eps^{-2}\right)$. The proof is complete.
\end{proof}

\begin{remark}
(a) In the case of Aygris finite element method, $r=5$
in \eqref{nest}--\eqref{nest4}.

(b) In the case of Legendre spectral method, $h=\frac{1}{N}$ in
\eqref{nest}--\eqref{nest4}, where $N$ stands for the highest degree of
polynomials in $V^N$.
\end{remark}

\section{Error Analysis for Galerkin Methods \eqref{galmethod} and
\eqref{spectralmethod}}\label{sec-4}

The goal of this section is to derive optimal order error estimates
for the Galerkin methods \eqref{galmethod} and \eqref{spectralmethod}.
Our main idea is to use a combined fixed point and linearization technique
which was introduced by the authors in \cite{Feng_Neilan_Prohl07}.
Once again, we only present the detailed analysis for \eqref{galmethod}
since the analysis for \eqref{spectralmethod} is essentially same.

First, we define a linear operator $T: V^h_g\to V^h_g$.
For any $w_h\in V^h_g$, let $T(w_h)\in V^h_g$ denote the solution of
following problem:
\begin{align}\label{Toperator}
B[w_h-T(w_h),\psi_h] &= \eps(\Delta w_h,\Delta \psi_h) -
(\text{det}(D^2 w_h),\psi_h) \\
&\hskip 1in +(f,\psi_h) - \left\langle\eps^2,
\normd{\psi_h}\right\rangle_{\partial\Omega}\quad \forall \psi_h\in V^h_{0}.
\nonumber
\end{align}

By Theorem \ref{existence}, we see $T(w_h)$ is uniquely defined.
Notice that the right-hand side of \eqref{Toperator} is the residual
of $w_h$ to equation \eqref{galmethod}.  It is easy to see that any
fixed point $w_h$ of the mapping $T$ (i.e. $T(w_h)=w_h$) is a
solution to problem \eqref{galmethod} and vice-versa. In the following we shall
show that the mapping $T$ indeed has a unique fixed point in a small
neighborhood of $I_h u^\eps$. To this end, we set
\[
B_h(\rho):=\bigl\{v_h\in V^h_g;\ \|v_h-I_hu^\eps\|_\htw \le \rho \bigr\}.
\]
In the rest of the section, we assume $u^\eps\in H^s(\Omega)$
and set $\ell=\text{min}\{r+1,s\}$.

\begin{lem}\label{lem51}
There exists a constant $C_7(\eps)=O\left(\eps^{-n}\right)>0\, (n=2,3)$ such that
\begin{equation}\label{eq51}
\|I_h u^\eps-T(I_h u^\eps)\|_{H^2} \le C_7(\eps) h^{\ell-2}\|u^\eps\|_{H^\ell}.
\end{equation}
\end{lem}

\begin{proof}
To simplify notation, let $\omega_h:=I_h u^\eps-T(I_h u^\eps)$ and
let $\eta_h:=I_hu^\eps-u^\eps$.  We then have
\begin{align*}
B[\omega_h,\omega_h]&=\eps(\Delta (I_h u^\eps),\Delta \omega_h)
-(\text{det}(D^2(I_h u^\eps))-f,\omega_h)
-\left\langle\eps^2,\normd{\omega_h}\right\rangle_{\partial\Omega}\\
&=\eps(\Delta \eta_h,\Delta \omega_h)+(\text{det}(D^2 u^\eps)
-\text{det}(D^2(I_hu^\eps)),\omega_h)\\
&\le \eps \|\Delta \eta_h\|_\lt||\Delta
\omega_h\|_\lt+\|\text{det}(D^2 u^\eps)
-\text{det}(D^2(I_h u^\eps))\|_\lt \|\omega\|_\lt.
\end{align*}
Using the Mean Value Theorem we have
\begin{align*}
\text{det}(D^2(I_hu^\eps))-\text{det}(D^2u^\eps)
&=\tilde{\Phi}:D^2\eta_h,
\end{align*}
where $\tilde{\Phi}=\text{cof}(\tau
D^2(I_hu^\eps)+(1-\tau)D^2u^\eps)$ for some $\tau\in [0,1]$.

Next, when $n=2$, we bound $\|\tilde{\Phi}\|_{L^\infty}$ as follows:
\begin{align*}
\|\tilde{\Phi}^\eps\|_{L^\infty}&=\|\text{cof}(\tau D^2(I_h u^\eps)
+(1-\tau)D^2 u^\eps)\|_{L^\infty} \\
&=\|\tau D^2(I_h u^\eps)+(1-\tau)D^2 u^\eps\|_{L^\infty}\\
&\le \|D^2(I_h u^\eps)\|_{L^\infty}+\|D^2u^\eps\|_{L^\infty} \\
&\le C \|D^2 u^\eps\|_{L^\infty}\le \frac{C}{\eps}.
\end{align*}
Similarly, when $n=3$, we can show that
$\|\tilde{\Phi}^\eps\|_{L^\infty}=O(\vepsi^{-2})$. Hence,
\begin{align*}
B[\omega_h,\omega_h]&\le \eps\|\Delta \eta_h\|_\lt \|\Delta
\omega_h\|_\lt + \frac{C}{\eps^{n-1}}\|D^2 \eta_h\|_\lt||\omega_h\|_\lt\\
&\le \frac{C}{\eps^{n-1}}\|\eta_h\|_\htw\|\omega_h\|_\htw.
\end{align*}
Using the coercivity of the bilinear form $B[\cdot,\cdot]$ we get
\[
\|\omega_h\|_\htw\le \frac{C}{\eps^{n-1} C_2(\eps)}\|\eta_h\|
\le \frac{C}{\eps^{n-1} C_2(\eps)}h^{\ell-2}\|u^\eps\|_{H^\ell}.
\]
Thus, \eqref{eq51} holds with $C_7(\eps)=\frac{C}{\eps^{n-1} C_2(\eps)}
=O\left(\eps^{-n}\right)$.
\end{proof}


\begin{lem}\label{lem52}
There exists an $h_0>0$ and $0<\rho_0<1$ such that for $h\le h_0$,
the mapping T is a contracting mapping in the ball $B_h(\rho_0)$
with a contraction factor $\frac12$.  That is, for any $v_h,w_h\in
B_h(\rho_0)$, there holds
\begin{equation} \label{e4.3}
\|T(v_h)-T(w_h)\|_{H^2}\le \frac12 \|v_h-w_h\|_{H^2}.
\end{equation}
\end{lem}

\begin{proof}
For any $\psi_h\in V^h_0$, using the definition of $T(v_h)$ and $T(w_h)$
we get
\begin{align*}
B[T(v_h)-T(w_h),\psi_h] 
&=\left(\Phi^\eps D(v_h-w_h),D\psi_h\right)+
\left(\text{det}(D^2v_h)-\text{det}(D^2w_h),\psi_h\right).
\end{align*}

Let $v_h^\mu,\ w_h^\mu$ denote the standard mollifications of $v_h$ and
$w_h$, respectively. Adding and subtracting these terms and using
the Mean Value Theorem yield
\begin{align*}
&B[T(v_h)-T(w_h),\psi_h] \\
&\quad =(\Phi^\eps(Dv_h-Dw_h),D\psi_h)+(\text{det}(D^2v_h)
-\text{det}(D^2w_h),\psi_h)\\
&\quad =(\Phi^\eps(Dv_h-Dw_h),D\psi_h)+(\text{det}(D^2v^\mu_h)
-\text{det}(D^2w^\mu_h),\psi_h)\\
&\hspace{1.75cm}+(\text{det}(D^2v_h)-\text{det}(D^2v_h^\mu),\psi_h)+(\text{det}(D^2w^\mu_h)-\text{det}(D^2w_h),\psi_h)\\
&\quad =(\Phi^\eps(Dv_h-Dw_h),D\psi_h)+(\Psi_h:(D^2v^\mu_h-D^2w^\mu_h),\psi_h)\\
&\hspace{1.75cm}+(\text{det}(D^2v_h)-\text{det}(D^2v_h^\mu),\psi_h)
+(\text{det}(D^2w^\mu_h)-\text{det}(D^2w_h),\psi_h),
\end{align*}
where $\Psi_h=\text{cof}(D^2v_h^\mu+\tau(D^2w_h^\mu-D^2v_h^\mu)),\ \tau\in [0,1]$.

Using Lemma $\ref{lem3.1}$ we have
\begin{align*}
&B[T(v_h)-T(w_h),\psi_h]\\
&\quad =((\Phi^\eps-\Psi_h)(Dv_h-Dw_h),D\psi_h)
+(\Psi_h(Dv_h-Dv_h^\mu),D\psi_h)\\
&\qquad +(\Psi_h(Dw_h^\mu-Dw_h),\psi_h)+(\text{det}(D^2v_h)
-\text{det}(D^2v_h^\mu),\psi_h)\\
&\qquad +(\text{det}(D^2w^\mu_h)-\text{det}(D^2w_h),\psi_h)\\
&\quad
\le C\Bigl\{ \|\Phi^\eps-\Psi_h\|_\lt\|v_h-w_h\|_\htw \|\psi_h\|_\htw
+\|\Psi_h\|_\lt \|\psi_h\|_\htw \bigl[\|v_h-v_h^\mu\|_\htw\\
&\qquad +\|w_h-w_h^\mu\|_\htw \bigr]
+\bigl[\|\text{det}(D^2v_h)-\text{det}(D^2v_h^\mu)\|_\lt\\
&\qquad +\|\text{det}(D^2w_h)-\text{det}(D^2w_h^\mu)\|_\lt \bigr] \|\psi_h\|_\lt \Bigr\},
\end{align*}
where we have used Sobolev's inequality.

Next, we derive an upper bound for $\|\Phi^\eps-\Psi_h\|_\lt$ when
n=2 as follows:
\begin{align*}
\|\Phi^\eps-\Psi_h\|_\lt&=\|\Phi^\eps-\text{cof}(D^2v_h^\mu
+\tau(D^2w_h^\mu-D^2v_h^\mu))\|_\lt\\
&=\|D^2u^\eps-(D^2v^\mu_h+\tau(D^2w_h^\mu-D^2v_h^\mu))\|_\lt\\
&\le \|D^2u^\eps-D^2(I_hu^\eps)\|_\lt + \|D^2(I_hu^\eps)-D^2v_h\|_\lt\\
&\quad + 2\|D^2v_h-D^2v_h^\mu\|_\lt
+\|D^2w_h-D^2w_h^\mu\|_\lt+\|D^2v_h-D^2w_h\|\\
&\le Ch^{\ell-2}\|u^\eps\|_{H^\ell}+ 2\rho_0+ 2\|v_h-v_h^\mu\|_\htw +
\|w_h-w_h^\mu\|_\htw.
\end{align*}

Similarly, when n=3, we can show
\begin{align*}\|\Phi^\eps-\Psi_h\|_\lt&\le \frac{C}{\eps}\bigl(h^{\ell-2}\|u^\eps\|_{H^\ell}+ 2\rho_0+ 2\|v_h-v_h^\mu\|_\htw +
\|w_h-w_h^\mu\|_\htw\bigr).\end{align*}

Using this result above, we have
\begin{align*}
&B[T(v_h)-T(w_h),\psi_h]\\
&\quad \le
\frac{C}{\eps^{n-2}}\Bigl\{\bigl[h^{\ell-2}\|u^\eps\|_{H^\ell}
+ \rho_0+ \|v_h-v_h^\mu\|_\htw + \|w_h-w_h^\mu\|_\htw\bigr]\|v_h-w_h\|_\htw \\
&\qquad +\|\Psi_h\|_\lt \bigl[\|v_h-v_h^\mu\|_\htw
+\|w_h-w_h^\mu\|_\htw\bigr]\\
&\qquad +\bigl[\|\text{det}(D^2v_h)-\text{det}(D^2v_h^\mu)\|_\lt
+\|\text{det}(D^2w_h)-\text{det}(D^2w_h^\mu)\|_\lt\bigr]\Bigr\}\|\psi_h\|_\htw.
\end{align*}
Setting $\mu\to 0$ yields
\[
B[T(v_h-T(w_h),\psi_h]\le
\frac{C}{\eps^{n-2}}\left(h^{\ell-2}\|u^\eps\|_{H^\ell} +
\rho_0\right)\|v_h-w_h\|_\htw\|\psi_h\|_\htw.
\]
Using the coercivity of the bilinear form $B[\cdot,\cdot]$ we get
\[
\|T(v_h)-T(w_h)\|_\htw\le
\frac{C}{C_2(\eps)\eps^{n-2}}\left(h^{\ell-2}\|u^\eps\|_{H^\ell}
+\rho_0\right)\|v_h-w_h\|_\htw.
\]
Choosing $\rho_0=\frac{C_2(\eps)\eps^{n-2}}{C}$ and
$h_0=\left(\frac{C_2(\eps)\eps^{n-2}}{C\|u^\eps\|_{H^\ell}}\right)^{\frac{1}{\ell-2}}$,
then for $h\le h_0$ we have
\[
\|T(v_h)-T(w_h)\|_\htw\le \frac12 \|v_h-w_h\|_\htw.
\]
The proof is complete.
\end{proof}


\begin{thm}\label{maintheorem}
Let $\rho_1=2C_7(\eps)h^{\ell-2}\|u^\eps\|_{H^\ell}.$ Then there exists
an $h_1>0$ such that for $h\le \text{\rm min}\{h_0,h_1\}$, there
exists a unique solution $u^\eps_h$ of \eqref{galmethod} in the
ball $B_h(\rho_1)$. Moreover, there exists a constant
$C_8(\eps)=O\left(\eps^{-n}\right)$ such that
\begin{equation} \label{e4.4}
\|u^\eps-u^\eps_h\|_{H^2}\le C_8(\eps) h^{\ell-2}\|u^\eps\|_{H^\ell},\quad
\ell=\text{\rm min}\{r+1,s\}.
\end{equation}
\end{thm}

\begin{proof}
Let $v_h\in B_h(\rho_1)$ and
$h_1=\bigl(\frac{C\eps^{2n-1}}{\|u^\eps\|_{H^\ell}}\bigr)^{\frac{1}{\ell-2}}$.
Then $h\le \text{\rm min}\{h_0,h_1\}$ implies that $\rho_1\le
\rho_0$.  Thus, using the triangle inequality and Lemmas
$\ref{lem51}$ and \ref{lem52} we have
\begin{align*}
\|I_h u^\eps - T(v_h)\|_{H^2}
&\le \|I_h u^\eps-T(I_hu^\eps)\|_{H^2}+ \|T(I_h u^\eps )-T(v_h)\|_{H^2}\\
&\le C_7(\eps)h^{\ell-2}\|u\|_{H^\ell}+\frac12 \|I_hu^\eps-v_h\|_{H^2}\le
\frac{\rho_1}{2}+\frac{\rho_1}{2}=\rho_1.
\end{align*}
Hence, $T(v_h)\in B_h(\rho_1)$.  In addition, from Lemma \ref{lem52}
we know that $T$ is a contracting mapping.  Thus, the Brouwer fixed
Theorem \cite{Gilbarg_Trudinger01} infers that $T$ has a unique fixed
point $u^\eps_h\in B_h(\rho_1)$, which is the unique solution to \eqref{galmethod}.

To get the error estimate, we use the triangle inequality to get
\begin{align*}
\|u^\eps-u^\eps_h\|_{H^2} &\le \|u^\eps-I_h u^\eps\|_{H^2}+\|I_hu^\eps-u^\eps_h\| \\
&\le Ch^{\ell-2}\|u\|_{H^\ell}+\rho_1=C_8(\eps)h^{\ell-2}\|u\|_{H^\ell},
\end{align*}
where
$C_8(\eps):=CC_7(\eps)=O\left(\eps^{-n}\right)$. The proof is complete.
\end{proof}


\begin{thm}
In addition to the hypothesis of Theorem \ref{maintheorem}, assume
that the linearized equation is $H^4$-regular with the regularity
constant $C_s(\vepsi)$.  Then there holds
\begin{equation} \label{e4.5}
\|u^\eps-u^\eps_h\|_{L^2}\le C_{9}(\eps)\Bigl[
\frac{h^\ell}{\sqrt{\vepsi}} \|u^\eps\|_{H^\ell}
+\eps^{2-n}C_8(\eps) h^{2\ell-4}\|u^\eps\|_{H^\ell}^2 \Bigr],
\end{equation}
where $C_9(\vepsi)= C_s(\vepsi) C_8(\vepsi)$.
\end{thm}

\begin{proof}
Let $e_h^\eps:=\ue-\uh$ and $\um$ denote a standard mollification
of $u_h^\eps$. It is easy to verify that $e_h^\eps$ satisfies
the following error equation:
\[
\eps(\Delta e_h^\eps,\Delta \psi_h)+(\text{det}(D^2u_h^\eps)
-\text{det}(D^2u^\eps),\psi_h)=0\spa \forall \psi_h\in V^h_0.
\]
Using the Mean Value Theorem and Lemma \ref{lem3.1} we have
\begin{align*}
0&=\eps(\Delta e_h^\eps,\Delta\psi_h)+(\text{det}(D^2\um)
-\text{det}(D^2\ue),\psi_h)+(\text{det}(D^2\uh)-\text{det}(D^2\um),\psi_h)\\
&=\eps(\Delta e_h^\eps,\Delta\psi_h)-(\tilde{\Phi}D(\um-\ue),D\psi_h)
+(\text{det}(D^2\uh)-\text{det}(D^2\um),\psi_h),
\end{align*}
where $\tilde{\Phi}=\text{cof}(D^2\um+\tau(D^2\ue-D^2\um)),\ \tau\in [0,1]$.

Next, The $H^4$-regular assumption implies that (cf. Theorem $\ref{linbound}$)
there exists a unique solution $\psi$ to the following problem:
\begin{alignat*}{2}
L_{u^\epsilon}(\psi)&=e_h^\eps &&\qquad \text{in}\ \Omega,\\
\psi&=0 &&\qquad \text{on}\ \partial\Omega, \\
\Delta\psi &=0  &&\qquad \text{on}\ \partial\Omega.
\end{alignat*}
Moreover, there holds
\begin{equation} \label{h4regular}
\|\psi\|_{H^4}\le C_s(\eps)\|e_h^\eps\|_\lt.
\end{equation}
Thus,
\begin{align*}
\|e_h^\eps\|_\lt^2
&=\langle e_h^\eps,e_h^\eps\rangle=\eps(\Delta e_h^\eps,\Delta \psi)
+(\Phi^\eps D\psi,D e_h^\eps)\\
&=\eps(\Delta e_h^\eps,\Delta (\psi-I_h\psi))+(\Phi^\eps De^\eps_h,D(\psi-I_h\psi)
+\eps(\Delta e_h^\eps,\Delta (I_h\psi))\\
&\qquad +(\Phi^\eps De_h^\eps,D(I_h\psi))-\eps(\Delta e_h^\eps,\Delta (I_h\psi_h))
-(\tilde{\Phi}D(\ue-\um),D(I_h\psi))\\
&\qquad -(\text{det}(D^2\uh) -\text{det}(D^2\um),I_h\psi)\\
&=\eps(\Delta e_h^\eps,\Delta (\psi-I_h\psi))+(\Phi^\eps De_h^\eps,D(\psi-I_h\psi))
\\
&\qquad +(\Phi^\eps De_h^\eps-\tilde{\Phi}D(\ue-\um),D(I_h\psi))
-(\text{det}(D^2\uh)-\text{det}(D^2\um),I_h\psi)\\
&=\eps(\Delta e_h^\eps,\Delta (\psi-I_h\psi))+(\Phi^\eps
De_h^\eps,D(\psi-I_h\psi))+((\Phi^\eps-\tilde{\Phi})De_h^\eps,D(I_h\psi))\\
&\qquad +(\tilde{\Phi}D(\um-\uh),D(I_h\psi))+(\text{det}(D^2\um)-\text{det}(D^2\uh),I_h\psi)\\
&\le \|\Delta e_h^\eps\|_\lt \|\Delta (\psi-I_h\psi)\|_\lt+ C\|\Phi^\eps\|_\lt
\|e_h^\eps\|_\htw\|\psi-I_h\psi\|_\htw\\
&\qquad +C\|\Phi^\eps-\tilde{\Phi}\|_\lt\|e_h^\eps\|_\htw\|I_h\psi\|_\htw+C\|\tilde{\Phi}\|_\lt\|\um-\uh\|_\htw\|I_h\psi\|_\htw\\
&\qquad +\|\text{det}(D^2\uh)-\text{det}(D^2\um)\|_\lt\|I_h\psi\|_\lt,
\end{align*}
where we have used Sobolev's inequality.  Continuing, we have
\begin{align}\label{e4.7}
\|e_h^\eps\|_\lt^2&\le C\Bigl\{ \eps h^2\|e_h^\eps\|_\htw
+h^2\|\Phi^\eps\|_\lt\|e_h^\eps\|_\htw
+\|\Phi^\eps-\tilde{\Phi}\|_\lt\|e_h^\eps\|_\htw\\
&\qquad +\|\tilde{\Phi}\|_\lt\|\um-\uh\|_\lt
+\|\text{det}(D^2\uh)-\text{det}(D^2\ue)\|_\lt \Bigr\} \|\psi\|_{H^4}. \nonumber
\end{align}

We now bound $\|\Phi^\eps-\tilde{\Phi}\|_\lt$ for n=2 as follows:
\begin{align} \label{e4.8}
\|\Phi^\eps-\tilde{\Phi}\|_\lt&=\|\text{cof}(D^2\ue)
-\text{cof}(D^2\um+\tau(D^2\ue-D^2\um))\|_\lt \\
&=\|D^2\ue-D^2\um+\tau(D^2\um-D^2\ue)\|_\lt \nonumber \\
&\le \|D^2\ue-D^2\uh\|_\lt + \|D^2\uh-D^2\um\|_\lt \nonumber \\
&\qquad +\|D^2\um-D^2\uh\|_\lt+\|D^2\uh-D^2\ue\|_\lt  \nonumber\\
&=2\|D^2\ue-D^2\uh\|_\lt+2\|D^2\uh-D^2\um\|_\lt. \nonumber
\end{align}

Similarly, when n=3, we have
\begin{align}\label{e4.8a}\|\Phi^\eps-\tilde{\Phi}\|_\lt&\le
\frac{C}{\eps}\bigl(\|D^2\ue-D^2\uh\|_\lt+\|D^2\uh-D^2\um\|_\lt\bigr).\end{align}

Substituting \eqref{e4.8}-\eqref{e4.8a} into \eqref{e4.7} we obtain
\begin{align*}
\|e_h^\eps\|_\lt^2&\le C\Big(\eps h^2\|e_h^\eps\|_\htw+
h^2\|\Phi^\eps\|_\lt\|e_h^\eps\|_\htw+
\eps^{2-n}(\|e_h^\eps\|_\htw+ \|\uh-\um\|_\htw)\|e_h^\eps\|_\htw\\
&\qquad + \|\tilde{\Phi}\|_\lt \|\um-\uh\|_\htw+
\|\text{det}(D^2\uh)-\text{det}(D^2\um)\|_\lt\Big)\|\psi\|_{H^4}.
\end{align*}
Setting $\mu\to 0$ and using \eqref{h4regular} yield
\begin{align*}
\|e_h^\eps\|_\lt^2&\le C\left(\eps h^2\|e_h^\eps\|_\htw+
h^2\|\Phi^\eps\|_\lt\|e_h^\eps\|_\htw+
\eps^{2-n}\|e_h^\eps\|_\htw^2\right)\|\psi\|_{H^4}\\
&\le C_s(\eps)\left(\eps  h^2\|e_h^\eps\|_\htw+
h^2\|\Phi^\eps\|_\lt\|e_h^\eps\|_\htw+
\eps^{2-n}\|e_h^\eps\|_\htw^2\right)\|e_h^\eps\|_{L^2}.
\end{align*}
Hence,
\begin{align*}
\|e_h^\eps\|_\lt&\le C_s(\eps)\bigl(\eps h^2\|e_h^\eps\|_\htw
+ h^2\|\Phi^\eps\|_\lt\|e_h^\eps\|_\htw+\eps^{2-n}\|e_h^\eps\|_\htw^2\bigr)\\
&\le C_s(\eps)\Bigl(\frac{h^2}{\sqrt{\eps}}\|e_h^\eps\|_\htw
+\eps^{2-n}\|e_h^\eps\|_\htw^2 \Bigr)\\
&\le C_s(\eps)\Bigl( h^\ell \frac{C_8(\eps)}{\sqrt{\eps}}\|u^\eps\|_{H^\ell}
+ \eps^{2-n}C_8(\eps)^2 h^{2\ell-4}\|u^\eps\|_{H^\ell}^2 \Bigr)\\
&= C_9(\eps)\Bigl[ \frac{h^\ell}{\sqrt{\vepsi}} \|u^\eps\|_{H^\ell}
+\eps^{2-n}C_8(\eps)h^{2\ell-4}\|\ue\|_{H^\ell} \Bigr],
\end{align*}
where $C_9:=C_s(\eps) C_8(\eps)$.
\end{proof}


\begin{remark}
It can be shown that the corresponding error
estimates for spectral Galerkin method \eqref{spectralmethod} are
\begin{align} \label{e4.9}
\|u^\vepsi-u^\vepsi_N\|_\htw &\leq   C_8(\eps) N^{2-\ell}\|u^\eps\|_{H^\ell}, \\
\|u^\eps-u^\eps_h\|_{L^2} &\le C_{9}(\eps)\Bigl[
N^{-\ell}\vepsi^{-\frac12} \|u^\eps\|_{H^\ell} +\eps^{2-n}C_8(\eps)
N^{4-2\ell}\|u^\eps\|_{H^\ell}^2 \Bigr], 
\end{align}
provided that $u^\vepsi\in H^s(\Ome)$. Where $\ell= \text{\rm min}\{N+1,s\}$. 
We refer the reader to \cite{Neilan_thesis} for a detailed proof.
\end{remark}

\section{Numerical Experiments and Rates of Convergence}\label{sec-5}

In this section, we provide several $2$-D numerical experiments to
gauge the efficiency of the finite element method developed in the
previous sections.  We numerically find the ``best'' choice of the
mesh size $h$ in terms of $\eps$, and rates of convergence for both
$u^0-u^\eps$ and $u^\eps-u^\eps_h$.  All tests given below are done
on the domain $\Omega=[0,1]^2$.  We refer the reader to
\cite{Feng2,Neilan_thesis} for more extensive $2$-D
and $3$-D numerical simulations.

\subsection*{Test 1:}
For this test, we calculate $\|u^0-u^\eps_h\|$ for fixed $h=0.009$,
while varying $\vepsi$ in order to approximate $\|u^\eps-u^0\|$.
We use Argyris elements and set to solve problem \eqref{galmethod}
with the following test functions
\begin{alignat*}{3}
&\text{(a). } u^0=e^{(x^2+y^2)/2},\quad
&&f=(1+x^2+y^2)e^{(x^2+y^2)/2},\quad &&g=e^{(x^2+y^2)/2}.\\
&\text{(b). } u^0=x^4+y^2,\quad &&f=24x^2,\quad
&&g=x^4+y^2.
\end{alignat*}

After having computed the error, we divide by various powers of
$\vepsi$ to estimate the rate at which each norm converges. Tables
$\ref{test1ab}$ and $\ref{test1bb}$ clearly show that
$\|u^0-u^\eps_h\|_\htw=O(\eps^{\frac14})$.  Since we have
fixed h very small, then $\|u^0-u^\eps\|_\htw\approx
\|u^0-u^\eps_h\|_\htw=O(\eps^{\frac14})$.  Based on this
heuristic argument, we predict that
$\|u^0-u^\eps\|_\htw=O(\eps^{\frac14})$.  Similarly, from
Tables \ref{test1ab} and \ref{test1bb}, we see that
$\|u^0-u^\eps\|_\lt\approx O(\eps)$ and $\|u^0-u^\eps\|_{H^1}\approx
O(\eps^\frac12)$.

\begin{table}[tbp]
\begin{center}
{\small \noindent\begin{tabular}{|l|l|l|l|}\hline
$\eps$ & $\|u_h^\eps-u^0\|_\lt$ & $\|u_h^\eps-u^0\|_{H^1}$ & $\|u_h^\eps-u^0\|_\htw$\\
\hline 0.75 &   0.109045862& 0.528560309 &3.39800721\\
\hline 0.5& 0.113340196& 0.548262711& 3.524120741\\
\hline 0.1& 0.08043631& 0.401646611& 3.071852861\\
\hline 0.075 &0.06932532& 0.352221521& 2.900677852\\
\hline 0.05& 0.053925875& 0.283684684 &2.657326288\\
\hline 0.025& 0.032202484& 0.18559903& 2.270039867\\
\hline 0.0125& 0.017972835& 0.117524466& 1.928506935\\
\hline 0.005& 0.007871272& 0.062721607& 1.544066061\\
\hline 0.0025& 0.004115832& 0.038522721& 1.301171395\\
\hline 0.00125& 0.002124611& 0.023464656& 1.095145652\\
\hline 0.0005&
0.00087474& 0.012073603& 0.871227869\\
\hline\end{tabular}} \caption{\label{test1a}{\scriptsize Test 1a:
Change of $\|u^0-u^\eps_h\|$ w.r.t. $\eps$ ($h=0.009$)}}
\end{center}
\end{table}

\begin{table}[tbp]
\begin{center}
{\small \noindent\begin{tabular}{|l|l|l|l|}\hline
$\eps$ & $\frac{\|u_h^\eps-u^0\|_\lt}{\vepsi}$ & $\frac{\|u_h^\eps-u^0\|_{H^1}}{\sqrt{\eps}}$ & $\frac{\|u_h^\eps-u^0\|_\htw}{\sqrt[4]{\eps}}$\\
\hline0.75 &   0.145394483 &0.610328873 &3.651396376\\
\hline0.5 &0.226680392 &0.775360561 &4.19090946\\
\hline0.1 &0.804363102 &1.270118105 &5.462612693\\
\hline0.075 &0.924337601 &1.286131149& 5.542863492\\
\hline0.05 &1.078517501 &1.268676476& 5.619560909\\
\hline0.025 &  1.288099375 &1.173831334 &5.708848032\\
\hline0.0125 &1.437826805 &1.051170776 &5.767580991\\
\hline0.005& 1.574254363& 0.887017474 &5.806619604\\
\hline0.0025 & 1.646332652& 0.770454411& 5.819015381\\
\hline0.00125 &1.699688442& 0.663680706& 5.82430863\\
\hline0.0005 &1.749480266 &0.53994796& 5.826251909\\
\hline
\end{tabular}} 
\caption{\label{test1ab}{\scriptsize Test 1a:
Change of $\|u^0-u^\eps_h\|$ w.r.t. $\eps$ ($h=0.009$)}}
\end{center}
\end{table}

\begin{table}[tbp]
\begin{center}
{\small \noindent\begin{tabular}{|l|l|l|l|}\hline
$\eps$ & $\|u_h^\eps-u^0\|_\lt$ & $\|u_h^\eps-u^0\|_{H^1}$ & $\|u_h^\eps-u^0\|_{H^2}$\\
\hline 0.75 &   0.179911089 &0.896016741 &5.98759668\\
\hline 0.5 & 0.177287901 &0.883816723& 5.982088348\\
\hline 0.1& 0.102586549& 0.549713562& 4.822739159\\
\hline 0.075 &0.085457592& 0.47264786& 4.537189438\\
\hline 0.05& 0.063960926& 0.374513017 &4.150185418\\
\hline 0.025 &0.036755952 &0.24464464& 3.552006757\\
\hline 0.0125& 0.020291198& 0.157714933& 3.032842066\\
\hline 0.005 &0.008967657& 0.087384209& 2.451390014\\
\hline 0.0025 &0.004761813& 0.055425626 &2.080704688\\
\hline 0.00125& 0.002502224& 0.034885527& 1.762183589\\
\hline 0.0005& 0.001054596& 0.018689724 &1.410593138\\
\hline 0.00025& 0.000544002& 0.011565172& 1.189359491\\
\hline 0.000125& 0.000279021& 0.007112& 0.999863491\\
\hline 0.00005& 0.000114659 &0.003700268 &0.787092117\\
\hline
\end{tabular}}
\caption{\label{test1b}{\scriptsize Test 1b:
Change of $\|u^0-u^\eps_h\|$ w.r.t. $\eps$ ($h=0.009$)}}
\end{center}
\end{table}

\begin{table}[tbp]
\begin{center}
{\small \noindent\begin{tabular}{|l|l|l|l|}\hline
$\eps$ & $\frac{\|u_h^\eps-u^0\|_\lt}{\vepsi}$ & $\frac{\|u_h^\eps-u^0\|_{H^1}}{\sqrt{\eps}}$ & $\frac{\|u_h^\eps-u^0\|_\htw}{\sqrt[4]{\eps}}$\\
\hline0.75  &  0.239881452 &1.034631013 &6.434091356\\
\hline0.5& 0.354575803 &1.249905596 &7.113942026\\
\hline0.1 &1.025865488 &1.738346916 &8.576177747\\
\hline0.075  & 1.139434558& 1.725865966& 8.670049892\\
\hline0.05 &1.279218511 &1.67487313&  8.776573597\\
\hline0.025 &  1.470238076& 1.547268561& 8.932824077\\
\hline0.0125&  1.623295808& 1.410645242 &9.070313375\\
\hline0.005 &1.793531488& 1.235799336 &9.218704868\\
\hline0.0025&  1.904725072 &1.108512517 &9.305194248\\
\hline0.00125 &2.001778889& 0.986711711& 9.371813749\\
\hline0.0005 &2.109192262& 0.835829887& 9.433204851\\
\hline0.00025 &2.176008824& 0.731445692& 9.458627896\\
\hline0.000125 &2.232164725& 0.636116593& 9.456125065\\
\hline0.00005& 2.293174219 &0.523296856 &9.360155452\\
\hline\end{tabular}} 
\caption{\label{test1bb} {\scriptsize Test 1a:
Change of $\|u^0-u^\eps_h\|$ w.r.t. $\eps$ ($h=0.009$)}}
\end{center}
\end{table}

\subsection*{Test 2}
The purpose of this test is to calculate the rate of convergence of
$||u^\eps-u^\eps_h||$ for fixed $\eps$ in various norms. As in Test
1, we use Argyris elements and solve problem \eqref{galmethod} with
boundary condition $\Delta u^\eps=\eps$ on $\partial\Omega$ being
replaced by $\Delta u^\eps=\phi^\eps$ on $\partial\Omega$.  We use the
following test functions:

\begin{alignat*}{2}
\text{(a). } &u^\eps=20x^6+y^6,\ \ &&f^\eps=18000x^4y^4-\eps(7200x^2+360y^2),\\
&g^\eps=20x^6+y^6,\ \ && \phi^\eps=600x^4+30y^4.\\
\text{(b.) } &u^\eps=x\text{sin}(x)+y\text{sin}(y),\quad
&&f^\eps=(2\text{cos}(x)-x\text{sin}(x))(2\text{cos}(y)-y*\text{sin}(y))\\
& &&\quad -\eps(x\text{sin}(x)-4\text{cos}(x)+y\text{sin}(y)-4\text{cos}(y)),\\
&g^\eps=x\text{sin}(x)+y\text{sin}(y),\quad
&&\phi^\eps=2\text{cos}(x)-x\text{sin}(x)+2\text{cos}(y)-y\text{sin}(y).
\end{alignat*} 

After recording the error, we divided each norm by the power of h
expected to be the convergence rate by the analysis in the previous
section. As seen by Tables $\ref{test2ab}$ and $\ref{test2bb}$, the
error converges quicker than anticipated in all the norms.

\begin{table}[tbp]
\begin{center}
{\small \noindent\begin{tabular}{|c|c|c|c|}\hline  $h$ &
$\|u^\eps-u^\eps_h\|_\lt$ & $\|u^\eps-u^\eps_h\|_{H^1}$ &
$\|u^\eps-u^\eps_h\|_\htw$\\
\hline
0.083333333& 4.09993E-05& 0.00268815 & 0.169852878\\
\hline 0.05&    1.08355E-06 &9.91661E-05& 0.011700487\\
\hline 0.030656967& 3.64957E-08& 5.42931E-06& 0.000986132\\
\hline0.023836565& 7.67076E-09& 1.29176E-06 &0.000312985\\
\hline 0.015988237& 4.51167E-10& 1.0898E-07 &4.04141E-05\\
\hline 0.012833175& 8.8807E-11& 2.43919E-08 &1.18929E-05\\
\hline
\end{tabular}}
\caption{\label{test2a}{\scriptsize Test 2a:
Change of $\|u^\eps-u^\eps_h\|$ w.r.t. $h$ ($\eps=0.001$)}}
\end{center}
\end{table}

\begin{table}[tbp]
\begin{center}
{\small \noindent\begin{tabular}{|c|c|c|c|}\hline $h$ &
$\frac{\|u^\eps-u^\eps_h\|_\lt}{h^6}$ &
$\frac{\|u^\eps-u^\eps_h\|_{H^1}}{h^5}$ &
$\frac{\|u^\eps-u^\eps_h\|_\htw}{h^4}$\\
\hline 0.083333333& 122.4232319& 668.897675&  3522.069268\\
\hline 0.05& 69.34721174& 317.3313846 &1872.077947\\
\hline
0.030656967& 43.96086573& 200.4928789 &1116.396482\\
\hline 0.023836565& 41.81926563& 167.8666007 &969.5028297\\
\hline 0.015988237& 27.01059961& 104.3140517& 618.4873284\\
\hline 0.012833175&
19.88119861& 70.07682598& 438.4809442\\
\hline
\end{tabular}}
\caption{\label{test2ab}{\scriptsize Test 2a: Change of 
$\|u^\eps-u^\eps_h\|$ w.r.t. $h$ ($\eps=0.001$)}}
\end{center}
\end{table}

\begin{table}[tbp]
\begin{center}
{\small \noindent\begin{tabular}{|c|c|c|c|}\hline $h$ &
$\|u^\eps-u^\eps_h\|_\lt$ & $\|u^\eps-u^\eps_h\|_{H^1}$ &
$\|u^\eps-u^\eps_h\|_\htw$ \\ \hline
0.083333333& 2.10771E-08& 4.91457E-07& 4.165E-05\\
\hline  0.05& 5.17295E-10& 1.90347E-08& 2.72117E-06\\
\hline
0.030656967& 1.77011E-11 &1.04548E-09 &2.39861E-07\\
\hline  0.023836565& 3.59463E-12& 2.62405E-10& 7.50769E-08
\\ \hline  0.015988237& 2.03078E-13& 2.23902E-11& 9.51713E-09\\
\hline  0.012833175& 3.64151E-14 &4.95624E-12&
2.69794E-09\\
\hline
\end{tabular}}
\caption{\label{test2b}{\scriptsize Test 2b:
Change of $\|u^\eps-u^\eps_h\|$ w.r.t. $h$ ($\eps=0.001$)}}
\end{center}
\end{table}

\begin{table}[tbp]\begin{center}
{\small \noindent\begin{tabular}{|c|c|c|c|}\hline $h$ &
$\frac{\|u^\eps-u^\eps_h\|_\lt}{h^6}$ &
$\frac{\|u^\eps-u^\eps_h\|_{H^1}}{h^5}$ &
$\frac{\|u^\eps-u^\eps_h\|_\htw}{h^4}$\\
\hline 0.083333333& 0.062935746& 0.122290283& 0.863654524\\
\hline 0.05&
0.033106867& 0.06091104&  0.435387754\\
\hline 0.030656967& 0.021321831& 0.038607272& 0.271545609\\
\hline 0.023836565& 0.019597137& 0.034099981& 0.232558269\\
\hline 0.015988237& 0.012157901& 0.021431653& 0.145647654\\
\hline 0.012833175&
0.008152235& 0.014239078& 0.099470776\\
\hline\end{tabular}}
\caption{\label{test2bb}{\scriptsize Test 2b:
Change of $\|u^\eps-u^\eps_h\|$ w.r.t. $h$ ($\eps=0.001$)}}
\end{center}
\end{table}

\subsection*{Test 3}
In this section, we fix a relation between $\vepsi$ and $h$ to
determine a "best" choice for $h$ in terms of $\vepsi$ such that the
global error $u^0-u^\eps_h$ is the same convergence rate as that of
$u^0-u^\eps$.  We solve problem \eqref{galmethod} with the following
test functions and parameters.

\begin{alignat*}{3}\text{(a). } &u^0=x^4+y^2,\quad  &&f=24x^2,\quad
&&g=x^4+y^2.\\
\text{(b). } &u^0=20x^6+y^6,\quad  &&f=18000x^4y^4,\quad
&&g=20x^6+y^6.\end{alignat*}

To see which relation gives the sought-after convergence rate, we
compare the data with a function, $y=\beta x^\alpha$, where
$\alpha=1$ in the $L^2$-case, $\alpha=\frac12$ in the $H^1$-case and
$\alpha=\frac14$ in the $H^2$-case. The constant, $\beta$, is
determined using a least squares fitting algorithm based on the data.

Figures \ref{fig5}--\ref{fig6} and \ref{fig1}--\ref{fig2} show that when
$h=\vepsi^{\frac12}$, $\|u^0-u^\eps_h\|_\htw\approx
O(\eps^{\frac14})$ and $\|u^0-u^\eps_h\|_\lt\approx
O(\eps)$. We can conclude from the data that the relation,
$h=\vepsi^{\frac12}$ is the "best choice" for h in terms of
$\vepsi$. It can also be seen from Figures \ref{fig3}--\ref{fig4}
that when $h=\eps$, $\|u^0-u^\eps_h\|_{H^1}\approx O(\eps^{\frac12})$.

\begin{figure}[h]
\begin{center}
\includegraphics[angle=0,width=6.25cm,height=5cm]{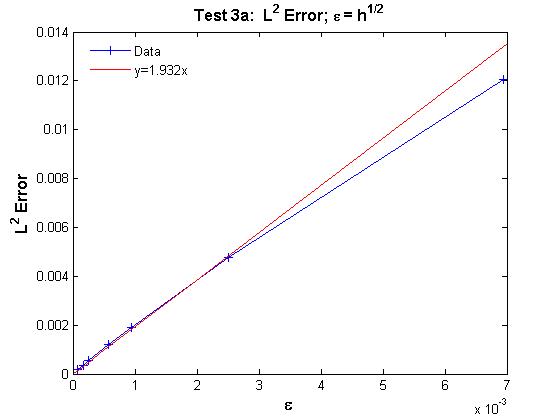}
\includegraphics[angle=0,width=6.25cm,height=5cm]{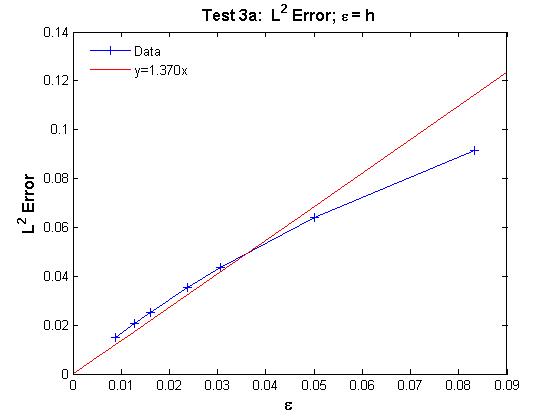}\\
\includegraphics[angle=0,width=6.25cm,height=5cm]{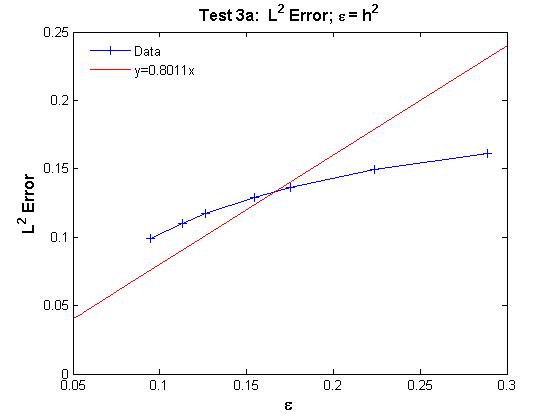}
\includegraphics[angle=0,width=6.25cm,height=5cm]{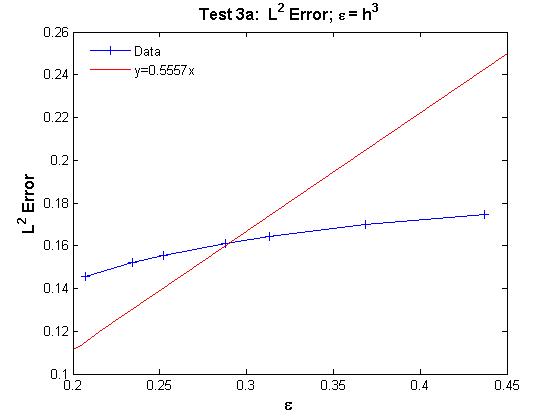}\end{center}
\caption{\label{fig1}{\scriptsize Test 3a.  $L^2$ Error}}\end{figure}

\begin{figure}[h]\begin{center}
\includegraphics[angle=0,width=6.25cm,height=5cm]{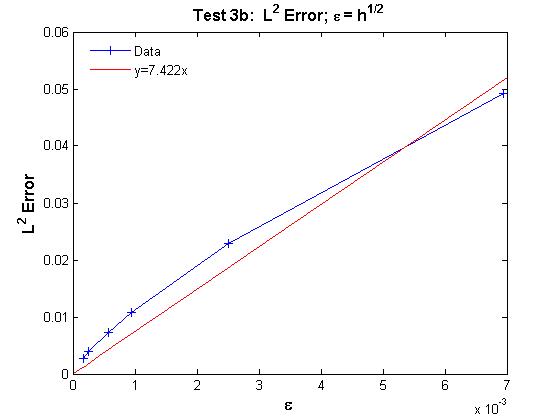}
\includegraphics[angle=0,width=6.25cm,height=5cm]{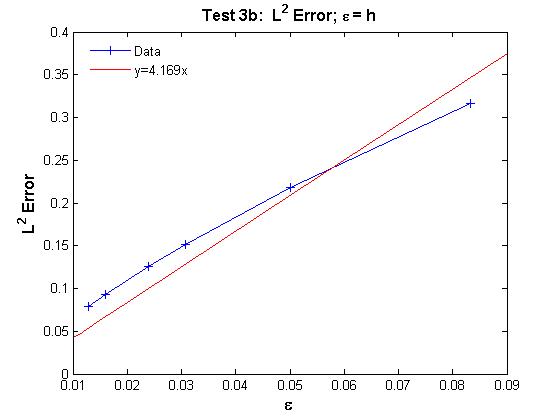}\\
\includegraphics[angle=0,width=6.25cm,height=5cm]{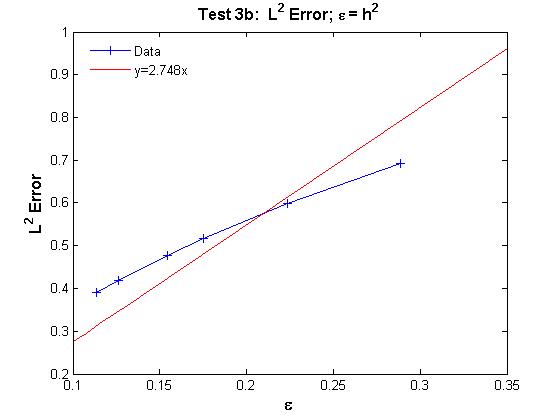}
\includegraphics[angle=0,width=6.25cm,height=5cm]{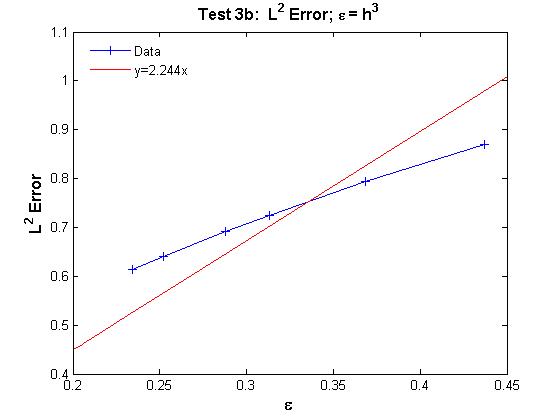}\end{center}
\caption{\label{fig2}{\scriptsize Test 3b.  $L^2$ Error}}\end{figure}

\begin{figure}[h]\begin{center}
\includegraphics[angle=0,width=6.25cm,height=5cm]{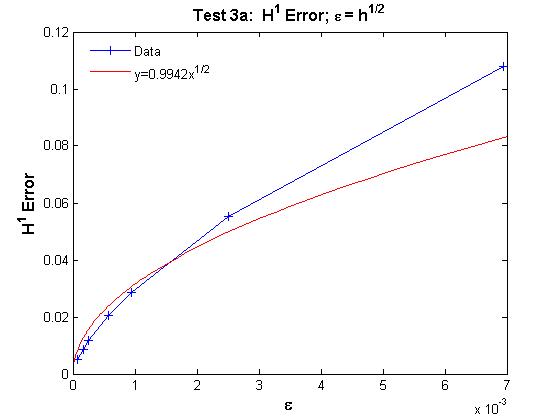}
\includegraphics[angle=0,width=6.25cm,height=5cm]{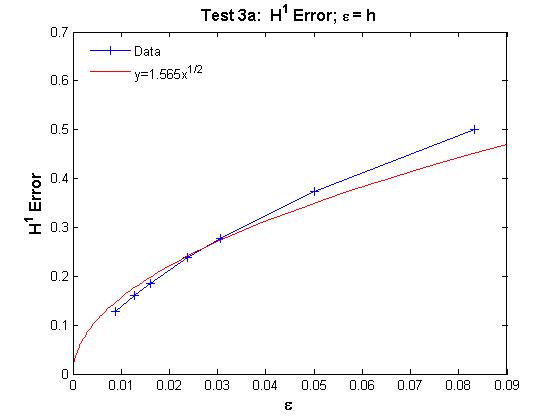}\\
\includegraphics[angle=0,width=6.25cm,height=5cm]{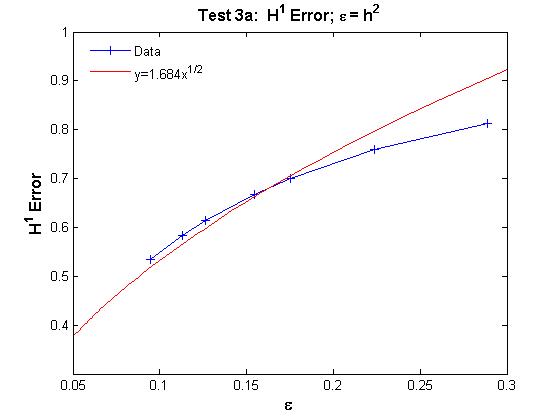}
\includegraphics[angle=0,width=6.25cm,height=5cm]{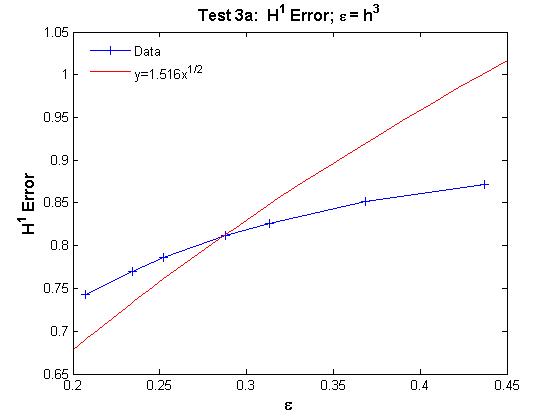}\end{center}
\caption{\label{fig3}{\scriptsize Test 3a.  $H^1$ Error}}\end{figure}

\begin{figure}[h]\begin{center}
\includegraphics[angle=0,width=6.25cm,height=5cm]{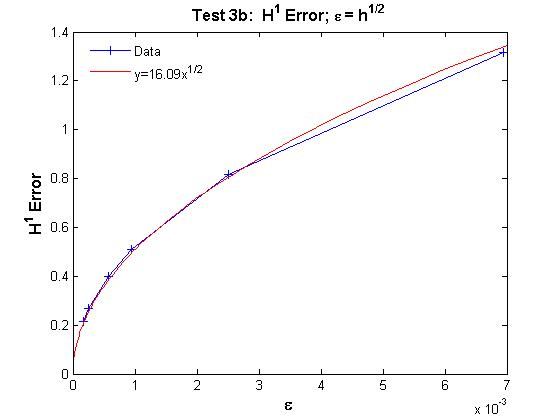}
\includegraphics[angle=0,width=6.25cm,height=5cm]{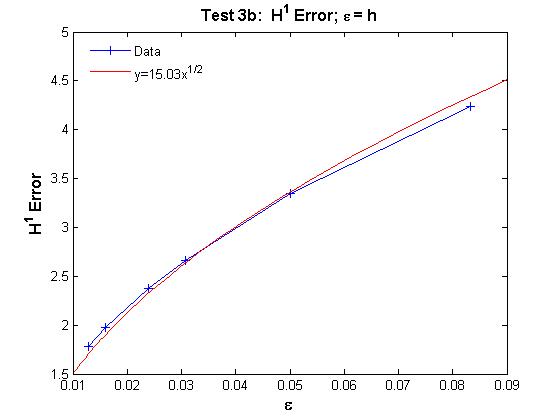}\\
\includegraphics[angle=0,width=6.25cm,height=5cm]{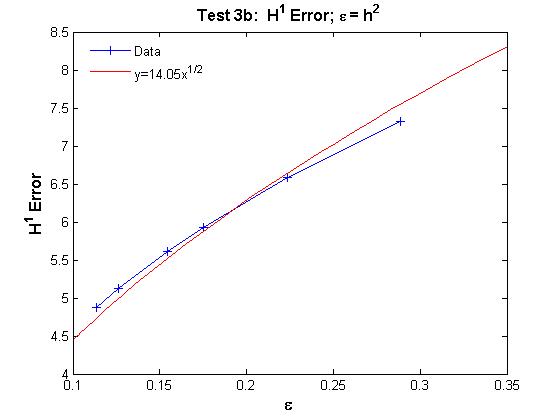}
\includegraphics[angle=0,width=6.25cm,height=5cm]{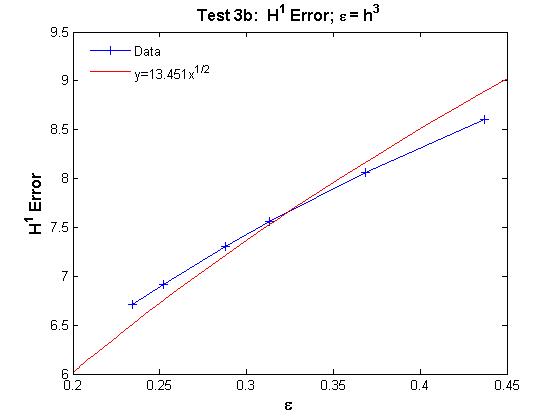}\end{center}
\caption{\label{fig4}{\scriptsize Test 3b. $H^1$ Error}}\end{figure}

\begin{figure}[h]\begin{center}
\includegraphics[angle=0,width=6.25cm,height=5cm]{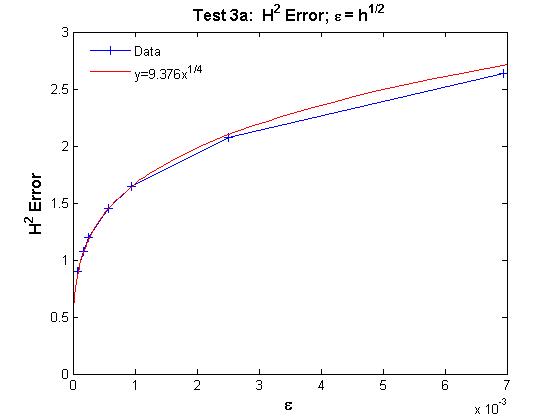}
\includegraphics[angle=0,width=6.25cm,height=5cm]{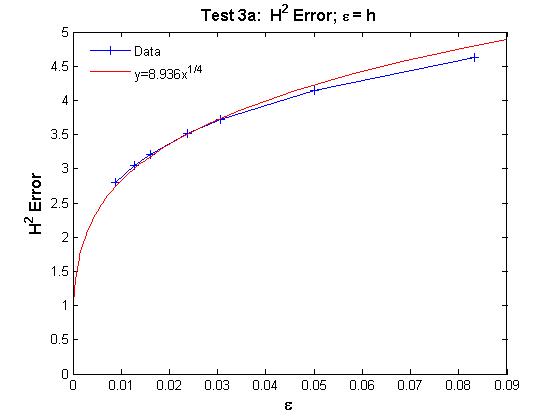}\\
\includegraphics[angle=0,width=6.25cm,height=5cm]{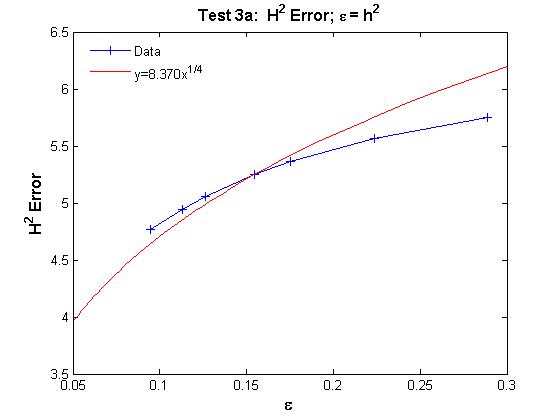}
\includegraphics[angle=0,width=6.25cm,height=5cm]{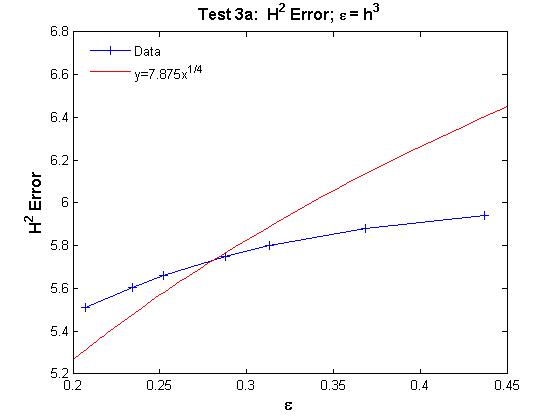}\end{center}
\caption{\label{fig5}{\scriptsize Test 3a. $H^2$ Error}}\end{figure}

\begin{figure}[h]\begin{center}
\includegraphics[angle=0,width=6.25cm,height=5cm]{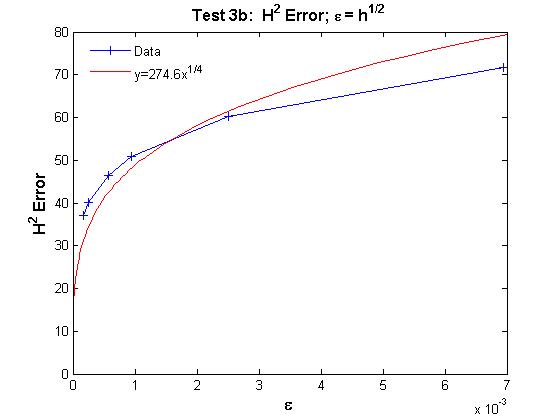}
\includegraphics[angle=0,width=6.25cm,height=5cm]{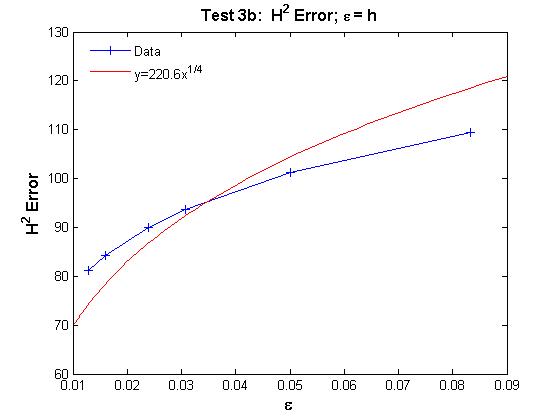}\\
\includegraphics[angle=0,width=6.25cm,height=5cm]{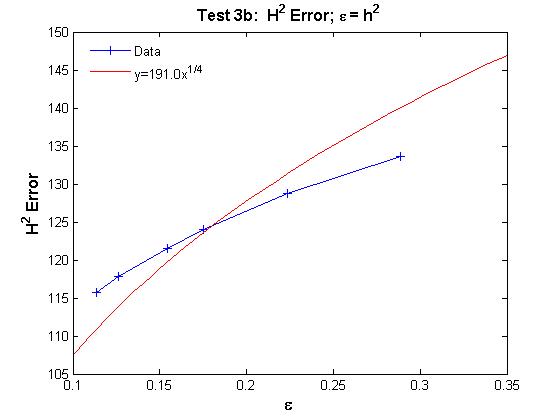}
\includegraphics[angle=0,width=6.25cm,height=5cm]{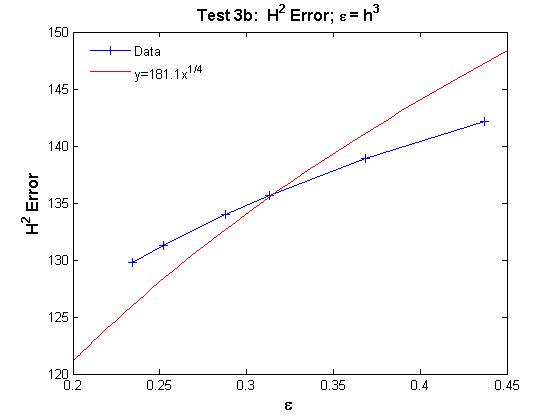}\end{center}
\caption{\label{fig6}{\scriptsize Test 3b.  $H^2$ Error}}
\end{figure}

\end{document}